\documentclass[10pt,a4paper]{article}
\usepackage[utf8]{inputenc}
\usepackage{tikz}
\usepackage[width=14.00cm, height=25.00cm]{geometry}
\usepackage[T1]{fontenc}
\usepackage{amsmath,amssymb,amsfonts,amsthm,amscd, bm, mathtools}
\usepackage{enumerate}
\usepackage{todonotes}
\usepackage{tikz}
\usepackage{color}
\usepackage[colorlinks=false]{hyperref}
\usepackage[capitalize]{cleveref}
\newcommand{\myref}[1]{\hyperref[#1]{\namecref{#1} \ref{#1}}}
\newcommand{\st}{\, : \,}

\usetikzlibrary{positioning}
\usetikzlibrary{decorations.pathmorphing}
\usetikzlibrary{decorations.text}
\newtheorem{thm}{Theorem}[section]
\newtheorem*{thm1}{Theorem \ref{t3}}
\newtheorem*{thm2}{Theorem \ref{t4}}
\newtheorem*{thm3}{Theorem \ref{generalcasegdelta}}
\newtheorem*{thm4}{Theorem \ref{end-counter-example}}
\newtheorem{prop}[thm]{Proposition}
\newtheorem{lemma}[thm]{Lemma}
\newtheorem{cor}[thm]{Corollary}
\newtheorem{defn}[thm]{Definition}

\newtheorem{ex}[thm]{Example}

\newtheorem{question}[thm]{Question}

\newcommand{\catname}[1]{{\normalfont\textbf{#1}}}

\title{A topological characterization of end space of infinite graphs via games, subspaces and products.}
\author{\normalsize Leandro Aurichi, Gustavo Boska, Davide Giacopello and Paulo Magalhães Júnior}

\newcommand{\Addresses}{{
  \bigskip
  \footnotesize
  L. ~Aurichi, \textsc{Instituto de Ci\^encias Matem\'aticas e de Computa\c c\~ao, Universidade de S\~ao Paulo\\
	Avenida Trabalhador s\~ao-carlense, 400,  S\~ao Carlos, SP, 13566-590, Brazil}\par\nopagebreak
  \textit{E-mail address}, L.~Aurichi: \texttt{aurichi@icmc.usp.br}

 \medskip
  G.~Boska, \textsc{Instituto de Ci\^encias Matem\'aticas e de Computa\c c\~ao, Universidade de S\~ao Paulo\\
	Avenida Trabalhador s\~ao-carlense, 400,  S\~ao Carlos, SP, 13566-590, Brazil}\par\nopagebreak
  \textit{E-mail address}, G.~Boska: \texttt{gustavo.boska@usp.br}

   \medskip
  D.~Giacopello, \textsc{MIFT Department, University of Messina, Italy}\par\nopagebreak
  \textit{E-mail address}, D.~Giacopello: \texttt{dagiacopello@unime.it}
  
  \medskip
  P.~Magalhães Júnior, \textsc{Instituto Federal do Rio Grande do Norte.}\par\nopagebreak
  \textit{E-mail address}, P.~Magalhães Júnior: \texttt{paulo.magalhaes@ifrn.edu.br}
 
}}

\date{}

\begin{document}

\maketitle

\begin{abstract}
In 1992, Diestel asked which topological spaces could be represented as the end space of some graph. In 2023, Pitz provided a solution to this question by giving a topological characterization of end spaces using a hereditarily complete special subbase. In this paper, we present an alternative topological characterization of end spaces, in which we employ a special subbase and a topological game. Furthermore, we provide several applications of this characterization: we show that every end space is hereditarily Baire, that $G_{\delta}$ subspaces of end spaces are also end spaces, and that the product of end spaces is not always an end space.

\end{abstract}

\section{Introduction}
\paragraph{}

In graph theory, a \textbf{ray} in an infinite graph is a one-way infinite path; furthermore, its infinite connected subgraphs are called \textbf{tails}. In 1964, Halin \cite{Halin} defined the following equivalence relation on the set of rays, $\mathcal{R}(G)$, of an infinite graph $G$: two rays $r$ and $s$ in an infinite graph $G$ are said to be equivalent, denoted by $r \sim s$, if for every finite set of vertices $X \subset V(G)$, $r$ and $s$ have tails in the same connected component of $G - X$. Based on this, Halin defined the \textbf{ends} of an infinite graph $G$ as the equivalence classes of $\mathcal{R}(G)/\sim$. This definition is inspired by other definitions of ends in different contexts, such as topology \cite{Freudenthal} and algebra \cite{Hopf}. The set of ends of an infinite graph $G$ is called the \textbf{end space} of $G$ and is denoted by $\Omega(G)$.

Given a finite set of vertices $X \subset V(G)$ and an end $\varepsilon \in \Omega(G)$, there exists a unique connected component $C$ of $G - X$ that contains a tail of every ray in $\varepsilon$. We denote this component by $C(X, \varepsilon)$. The set $\Omega(G)$ is commonly viewed as a topological space; this is achieved by defining the basic open sets around an end $\varepsilon \in \Omega(G)$ as follows: given a finite set $X \subset V(G)$, we define $\Omega(X, C) = \{ \varepsilon' \in \Omega(G) : C = C(X, \varepsilon') \}$. Endowed with the topology generated by such sets, the end space of a graph $G$ is automatically Hausdorff and zero-dimensional.

The ends of a graph can be understood as the directions in which the graph extends to infinity. In the case of locally finite graphs, by endowing the graph with a 1-complex topology, its ends yield the Freudenthal compactification of the graph. Consequently, numerous applications of ends arise, such as in the extension of results from finite graphs (see \cite{locallyfinitegraphswithendsI, cycle-cocycle}); in problems within Geometric Group Theory, such as Bass-Serre theory, where Stallings's theorem uses ends to detect product structures in groups, such as amalgamated free products or HNN-extensions (see \cite{stallings}); and in the theory of Iterated Function Systems (IFS), where the fractal attractor is often identified as the end space of an infinite tree or an associated Cayley graph. In this context, the topological properties of the ends dictate the geometry of the limit set of the dynamical system (see \cite{BARNSLEY1993115, endsofschreier, attractorIFS}).

In 1992, Diestel asked in \cite{Diestelquestion} which topological spaces could be represented as the end space of some graph. Towards answering this question, many papers have been published studying the topological properties of end spaces. In 2021, Kurkofka and Pitz presented the first version of the paper \cite{representacao}, in which they prove a representation theorem for graph end spaces, reducing the study of these spaces to the ray spaces of special order trees. To state it, we will need some concepts from the theory of order trees.

Given $(T, \leq)$, we denote, as usual, $\lfloor t \rfloor \coloneqq \{s \in T \mid t \leq s\}$ and $\lceil t \rceil \coloneqq \{s \in T \mid s \leq t\}$. If $r < t$ and there is no $s \in T$ such that $r < s < t$, we say that $t$ is a successor of $r$. An order-theoretic tree is a poset such that $\lceil t \rceil \setminus \{t\}$ is well-ordered, and therefore order-isomorphic to an ordinal (its order type) called the \textbf{height} of $t$. If this is a limit ordinal, we say $t$ is a \textbf{limit element} or has \textbf{limit height}.  The height of the tree is the supremum of the heights of its elements. Notice that if $T$ has height $\alpha$, then for every $\beta < \alpha$, there must be an element $t \in T$ of height $\beta$. An \textbf{antichain} of a poset $(T, \leq)$ is a subset of $T$ whose elements are pairwise incomparable. We say a tree is \textbf{special} if its elements can be covered by a countable union of antichains.  We define \textbf{rays} as $\mathcal{R}(T) \coloneqq \{P \subseteq T : P \text{ is a chain with no maximal element}\}$. A ray $\omega$ may have (possibly infinite) \textbf{tops}, which are elements $t \in T$ such that $\lceil t \rceil \setminus \{t\} = \omega$. The topology of the ray space of a special tree is described by declaring the basic open sets around an element $r \in \mathcal{R}(T)$ as follows:

\begin{equation} \label{intrinseco}
    [t,F] = \{s \in \mathcal{R}(T): t\in s\text{ and }t'\notin s\text{ for every }t'\in F\},
\end{equation} where $t\in r$ and $F\subset T$ is a finite collection of tops of $r$. 
\begin{thm}[Kurkofka and Pitz \cite{representacao}]
    Let $X$ be a topological space. $X$ is homeomorphic to the end space of a graph if, and only if, $X$ is homeomorphic to the ray space of a special tree.
\end{thm}

From this representation theorem, it is possible to derive several fundamental properties of the end space based on the ray space of a special tree. The most significant are topological properties concerning a subbase of closed and open sets (clopens): being $\sigma$-disjoint, nested, and Noetherian. When a clopen subbase possesses all three of these properties and contains the entire space as an element, it is called a \textbf{special subbase}. By combining these subbase properties with an additional completeness condition — namely, that $\mathcal{C}$ is \textit{hereditarily complete} — Pitz \cite{pitz} solved Diestel's question with the following topological characterization of ray spaces of special trees:

\begin{thm}[Pitz \cite{pitz}]
    The following are equivalent for a Hausdorff space X:
    \begin{enumerate}
        \item $X$ is homeomorphic to the ray space of a special tree;
        \item $X$ admits a hereditarily complete, nested, clopen subbase that is noetherian and $\sigma$-disjoint.
    \end{enumerate}
\end{thm}

In this paper, we present a new topological characterization of the end spaces of infinite graphs. In our approach, we retain the requirement of a special subbase; however, we replace the completeness property with a topological game played on the space $X$, called the \textit{end-game}. We define the \textbf{end game} (denoted by $End_{\mathcal{B}}$) over a zero-dimensional topological space $X$ with a fixed basis $\mathcal{B}$ consisting of clopen subsets. Players I and II alternate their moves according to the following rules:
\begin{itemize}
    \item The Player I starts the game by declaring a basic open set $U_0\subset X$ and the Player II answers with an open cover $\mathcal{U}_0$ for $U_0$ consisting of pairwise disjoint basic open sets;
 \item At the $n$-th round of the game, with $n\geq 1$, Player I declares a basic open set $U_n$ that is contained in some $V_{n-1}\in \mathcal{U}_{n-1}$. Then, Player II answers with an open cover $\mathcal{U}_n$ for $U_n$ whose elements are pairwise disjoint basic open sets.
\end{itemize}

The assumption that $X$ is a zero-dimensional space guarantees that Player II has always a possible answer for moves of Player I. After countably many innings, enumerated in $\mathbb{N}$, Player II wins if there are unique $x\in X$ and $A\subset X\setminus \{x\}$ an open set such that $\bigcap_{n\in \mathbb{N}} U_n = \{x\}\cup A$, otherwise Player I wins. If Player II has a (stationary) wining strategy for $End_{\mathcal{B}}$, we write $II\uparrow End_{\mathcal{B}}(X)$. Similarly, $I\uparrow End_{\mathcal{B}}(X)$ means that Player I has a
(stationary) winning strategy. Within these definitions, we prove in the Section \ref{prova2} the following result:

\begin{thm1}
    Let $X$ be a Hausdorff topological space. Then, $\mathrm{II}\uparrow \mathrm{End}_{\mathcal{B}}(X)$ for some special basis $\mathcal{B}$ of $X$ if, and only if, $X$ is the end space of some graph.
\end{thm1}

As an application of our characterization, we obtain several results concerning the topological properties of end spaces. In Section \ref{choquetproperty}, we further explore the end-game and its connection to the famous Banach-Mazur game, which is closely related to the Baire property. From this, we derive results regarding the ray and branch spaces with respect to the Choquet and productively Baire properties. Let $T$ be an order tree. A subset $P \subset T$ is a \textbf{path} if it is a downward closed chain. Let $\mathcal{P}(T)$ be the set of paths, the collection of every path of $T$, we can naturally consider the subspace topology $\mathcal{P}(T) \subset 2^T$. The basic sets here are given by the product $2^T \approx \prod_{t \in T}\{0,1\}$. This is a closed set from a Hausdorff compact, itself a compact space. A \textbf{branch} is a maximal chain (also a path) of $T$. If every branch of the tree has the tree height, we say $T$ is \textbf{pruned}. Both the branch set $\mathcal{B}(T)$ and the ray set $\mathcal{R}(T)$ are naturally subsets of $\mathcal{P}(T)$ and therefore have well defined subspace topology. 

\begin{thm2}
   Ray and Branch spaces are Choquet, hence productively Baire.
\end{thm2}

In Section \ref{subspaces}, we study the problem of determining which subspaces of an end space are themselves end spaces. In \cite{pitz}, Pitz proves that closed subsets of an end space are also end spaces. In \cite{approximatinginfinitegraphsbynormaltrees}, Kurkofka, Melcher, and Pitz prove that open subsets of an end space are end spaces as well. Although some partial answers exist, the problem has not yet been fully resolved.  In Section \ref{subspaces}, we obtain a characterization for the metrizable case in Theorem \ref{metcharacterization}: precisely the $G_\delta$ subsets of an end space are end spaces of some graph. In particular, a Bernstein set of the Cantor space is a Baire subspace of an end space that is not itself an end space. Furthermore, towards contributing to the solution of the problem in the general case, we prove the following result as an application of our topological characterization of end spaces:

\begin{thm3}
Let $A \subset \Omega(G)$ be a $G_{\delta}$ subset, then $A$ is an end-space.
\end{thm3}

In Section \ref{product}, we study the closure of the class of end spaces under topological products. In Subsection \ref{metrizablecaseproduct}, the metrizable case is addressed. By connecting established combinatorial and topological results, we conclude that metrizable end spaces (as well as branch spaces) are closed under countable products. Example \ref{naturalpower} demonstrates that the countability hypothesis is necessary. In Subsection \ref{obstructions}, we examine the general case of products between end spaces and branch spaces. Example \ref{michael} provides a product of branch spaces that is not itself a branch space. Furthermore, as an application of our topological characterization of end spaces, we prove that the product of two end spaces is not always an end space:

\begin{thm4}
    End-spaces are not closed under finite products.
\end{thm4}

\section{End spaces via a topological game} \label{prova2}

\paragraph{}

As pointed out in the introduction, Diestel's question regarding the topological characterization of end spaces was only recently solved by Pitz in \cite{pitz}. In this section, we give a new characterization theorem via the topological game previously mentioned. This approach is reasonable since games codify some ordered trees—the objects that, under the appropriate conditions obtained in \cite{representacao}, represent end spaces faithfully.

We start this discussion by detailing the definition of the end game and explaining its motivation.  As a big picture, Theorem 3 of Kurkofka and Pitz in \cite{representacao} constructs an order tree that displays the topological behavior of the ends in a given graph $G$. 

\begin{thm}\label{t3}
    Let $X$ be a Hausdorff topological space. Then, $\mathrm{II}\uparrow \mathrm{End}_{\mathcal{B}}(X)$ for some special basis $\mathcal{B}$ of $X$ if, and only if, $X$ is the end space of some graph.
\end{thm}

Our Theorem \ref{t3} approaches the dual direction: for a topological space $X$ with suitable hypothesis over a basis $\mathcal{B}$, we aim to define a special order tree $T$ whose high-ray space is $X$. The nodes of $T$ will be some open basic sets from $\mathcal{B}$, organized throughout $T$ so that a high-ray correspond to a $\subseteq-$decreasing sequence $\{U_n\}_{n\in\mathbb{N}}\subset \mathcal{B}$. In this direction, regarding the basis $\mathcal{B}$, we recall that the \textbf{end game} $\mathrm{End}_{\mathcal{B}}$ between Players $\mathrm{I}$ and $\mathrm{II}$ is described by the following rules:

\begin{itemize}
    \item The Player $\mathrm{I}$ starts the game by declaring a basic open set $U_0\in \mathcal{B}$. Assuming that $X$ is \textit{zero-dimensional}, as it is every end space, Player $\mathrm{II}$ is able to answer with an open cover $\mathcal{U}_0\subset \mathcal{B}$ for $U_0$ composed by pairwise disjoint elements;
    \item At the $n-$th round of the game, with $n\geq 1$, Player $\mathrm{I}$ declares an open set $U_n \in \mathcal{B}$ that is contained in some $V_{n-1}\in \mathcal{U}_{n-1}$. As in the previous item, Player $\mathrm{II}$ answers with a pairwise disjoint family $\mathcal{U}_n$ of basic open sets that cover $U_n$.  
\end{itemize}

Hence, at the end of the game, the moves of Player $\mathrm{I}$ describe a $\subseteq-$decreasing-sequence $\{U_n\}_{n\in\mathbb{N}}$. Therefore, we call Player $\mathrm{II}$ the winner of the game if this choice of moves captures a point $x\in X$ in the following sense: the intersection $\displaystyle \bigcap_{n\in\mathbb{N}}U_n$ can be uniquely written as a disjoint union $\{x\}\cup A$, where $A\subset X$ is an open set not containing $x$. When the end game is played in a topological space $X$ without requiring all open sets to belong to a fixed basis, we denote it by $End(X)$.

Comparing this criteria with the open basic neighborhood in \myref{intrinseco}, $x$ will be further identified with a high-ray $r$ of some high-ray space, so that $A$ will correspond to the open set $\displaystyle\bigcup_{\substack{t\text{ top }\\ \text{of } r}}[t,\emptyset]$.

In our context, a \textbf{strategy} for Player $\mathrm{I}$ is a map $\varphi$ that declares an answer $U_{n+1} = \varphi(\mathcal{U}_0, \mathcal{U}_1,\dots,\mathcal{U}_n)$ for every sequence of moves $\{\mathcal{U}_i\}_{i=0}^n$ made by Player $\mathrm{II}$. Analogously, $\psi$ is a strategy for Player $\mathrm{II}$ if, for every sequence of moves $(U_0,U_1,\dots,U_n)$ made by Player $\mathrm{I}$, the open cover $\mathcal{U}_{n} = \psi(U_0,U_1,\dots,U_n)$ is an answer given by Player $\mathrm{II}$ at the $n-$th round of the game. Moreover, if $\psi$ depends only on the previous move of Player $\mathrm{I}$, we denote $\psi(U_n) = \psi(U_0,U_1,\dots,U_n)$ and call $\psi$ a \textbf{stationary} strategy. In this case, we say that $\psi$ is \textbf{winning} if Player $\mathrm{II}$ wins every match of the form $$\langle U_0,\psi(U_0),U_1,\psi(U_1),U_2,\psi(U_2),\dots\rangle,$$ where $\{U_n\}_{n\in\mathbb{N}}$ is the sequence of moves played by Player $\mathrm{I}$. When such a strategy exists, we denote $\mathrm{II}\uparrow \mathrm{End}_{\mathcal{B}}$. Similar definitions for stationary and winning strategies for Player $\mathrm{I}$ can be settled, although we will not need them in our study. Actually, unless opposite mentions, all strategies in this paper are assumed to be stationary.

Intuitively, the winning criteria for Player $\mathrm{II}$ suggests that a winning strategy for this player is a map that displays a wide range of points in $X$. If $X$ is taken to be the end space of some graph, this motivates the characterization given by Theorem \ref{t3}. Before proving it, we will precise some notations from its statement. Then, for a topological space $X$, we call a clopen subbase $\mathcal{C} = \{U_{\alpha}\}_{\alpha\in \Lambda}$ \textbf{special} if the following properties hold:

\begin{itemize}
    \item $\mathcal{C}$ is $\sigma-$\textbf{disjoint}, in the sense that it can be written as a countable union of \textit{antichains}, i.e., families of pairwise disjoint open sets;
    \item $\mathcal{C}$ is \textbf{nested}, namely, given $\alpha,\beta \in \Lambda$, then either $U_{\alpha}\cap U_{\beta} =\emptyset$, or $U_{\alpha}\subseteq U_{\beta}$, or $U_{\beta}\subseteq U_{\alpha}$;
    \item $\mathcal{C}$ is \textbf{noetherian}, i.e., every $\subseteq-$increasing chain of elements of $\mathcal{C}$ has a $\subseteq-$maximum element.
\end{itemize}

If such a clopen subbase exists, we call $X$ a \textbf{special} topological space. Then, the end space characterization described by Pitz in \cite{pitz} claims that $X$ is the end space of some graph if, and only if, it admits a clopen special subbase that is \textit{hereditarily complete}. In our approach via topological games, this last hypothesis, to be further mentioned, will be replaced by the assumption that Player $\mathrm{II}$ wins when playing the end game in a special basis. Here, a \textbf{special basis} for $X$ means the basis generated by some clopen special subbase. In that direction, the following description is useful:

\begin{lemma}\label{descricao}
    Let $X$ be a topological space and $\mathcal{C} = \{U_{\alpha}\}_{\alpha\in \Lambda}$ be a $\sigma-$disjoint, nested and noetherian collection of clopen sets. Then, $\mathcal{C}$ is a special subbase of $X$ if, and only if, 
    \begin{equation}\label{descricaobase}
        \mathcal{B} = \left\{U_{\alpha}\setminus \bigcup_{\beta\in F}U_{\beta} : U_{\alpha},U_{\beta} \in \mathcal{C}, F\subset \Lambda \text{ finite}\right\}
    \end{equation} is a basis for $X$. In particular, the elements of $\mathcal{B}$ are also clopen sets.
\end{lemma}
\begin{proof} We will denote by $U^c: = X\setminus U$ the complement of a subset $U$ in $X$, also inspiring the notation $\mathcal{C}^c = \{U_{\alpha}^c\}_{{\alpha}\in \Lambda}$. Suppose first that $\mathcal{C}$ is a clopen subbase for $X$, i.e., that $\mathcal{C}\cup \mathcal{C}^c$ is an open subbase. Hence, $\left\{\displaystyle \bigcap_{\alpha \in F_1}U_{\alpha} \cap \bigcap_{\beta\in F_2}U_\beta^c : F_1,F_2\subset \Lambda \text{ finite}\right\}$ is a basis for $X$. However, if $F_1,F_2\subset \Lambda$ are finite and $\displaystyle \bigcap_{\alpha\in F_1}U_{\alpha} \neq \emptyset$, then $\{U_{\alpha}\}_{\alpha \in F_1}$ is totally ordered by $\subseteq$, since $\mathcal{C}$ is nested. Therefore, we can write $\displaystyle\bigcap_{\alpha \in F_1}U_{\alpha} \cap \bigcap_{\beta\in F_2}U_{\beta}^c = U_{\gamma}\cap \bigcap_{\beta\in F_2}U_{\beta}^c = U_{\gamma}\setminus \bigcup_{\beta\in F_2}U_{\beta}, $ where $U_{\gamma}$ is the $\subseteq-$minimum element from $\{U_{\alpha}\}_{\alpha \in F_1}$. Then, $\mathcal{B}$ is also a basis for $X$.

Conversely, suppose that $\mathcal{B}$ is a basis for $X$. For a basic open set $B = \displaystyle U_{\alpha}\setminus\bigcup_{\beta\in F}U_{\beta} \in \mathcal{B}$, we can write $\displaystyle U_{\alpha}\setminus \bigcup_{\beta\in F}U_{\beta} = U_{\alpha}\cap \bigcap_{\beta\in F}U_{\beta}^c$, so that $B$ is a finite intersection of elements from $\mathcal{C}\cup \mathcal{C}^c$. Therefore, $\mathcal{C}$ is a clopen subbase for $X$.

\end{proof}

The basis description provided by \myref{descricaobase} is quite similar to the basic open sets for ray spaces of trees in \myref{intrinseco}. Due to this, an open set of the form $U_{\alpha}\setminus \displaystyle \bigcup_{\beta\in F}U_{\beta}$ will now be denoted by $[U_{\alpha},F]$, for every $\alpha\in \Lambda$ and every finite $F\subset \Lambda$. Relying on this comparison, one can verify that $\mathcal{R}(T)$ is a special topological space for every special tree $T$:

\begin{lemma}\label{52}
    If $T$ is a special tree, then $\mathcal{C} = \{[t,\emptyset]\}_{t\in T}$ is a special clopen subbase for $\mathcal{R}(T)$. In particular, the end space of any graph is a special topological space. 
\end{lemma}
\begin{proof}
    We will first observe that, for any $t\in T$, the open set $[t,\emptyset]$ is also closed in $\mathcal{R}(T)$. Indeed, if $r\in \mathcal{R}(T)$ is a high-ray that does not contain $t$, then either $t \geq t_0$ for some top $t_0$ of $r$ or $s\in r$ for some node $s\in T$ incomparable with $t$. In the former case, $r\in [x,\{t_0\}]\subset \mathcal{R}(T)\setminus [t,\emptyset]$ for any $x\in r$, while, in the latter, $r\in [s,\emptyset]\subset \mathcal{R}(T)\setminus [t,\emptyset]$. In both case, it is verified that $[t,\emptyset]$ is a closed set of $\mathcal{R}(T)$. In other words, $\mathcal{C}$ is a family of clopen sets.

    For nodes $s,t\in T$, we have the following relations:
    \begin{itemize}
        \item If $s \leq t$, then every high-ray containing $t$ also contains $s$, so that $[t,\emptyset]\subset [s,\emptyset]$;
        \item If $s$ and $t$ are incomparable, then no high-ray containing $s$ contains $t$, so that $[t,\emptyset]\cap [s,\emptyset] = \emptyset$.
    \end{itemize}
    In particular, $\mathcal{C}$ is a nested family. Moreover, if $\mathcal{D} = \{[t,\emptyset]\}_{t\in \Lambda}$ is a subset of $\mathcal{C}$ totally ordered by $\subseteq$, the above items show that $\Lambda \subset T$ is totally ordered as well. Then, there exists $t_0 = \min \Lambda$, so that $[t,\emptyset]\subset [t_0,\emptyset]$ for every $t\in T$. In other words, $\mathcal{D}$ is noetherian.

    Since $T$ is special, we can write $T = \displaystyle \bigcup_{n \in \mathbb{N}}A_n$, where $A_n \subset T$ is an antichain for each $n \in\mathbb{N}$. By the second item above, $\mathcal{A}_n = \{[t,\emptyset]:t\in T\}$ is also an antichain in $\mathcal{R}(T)$. Therefore, $\mathcal{C}$ is $\sigma-$disjoint, since $\mathcal{C} = \displaystyle \bigcup_{n\in\mathbb{N}}\mathcal{A}_n$.

    Finally, given $r\in \mathcal{R}(T)$, a node $t\in r$ and a finite set $F$ of tops of $r$, we have $[t,F] = [t,\emptyset]\setminus \displaystyle \bigcup_{s\in F}[s,\emptyset]$. In particular, the set 
    
    \begin{equation}\label{especial}
        \mathcal{B} = \left\{[t,\emptyset]\setminus \displaystyle \bigcup_{s\in F}[s,\emptyset] : t\in T, F\subset T\text{ finite}\right\}
    \end{equation} is a basis for $\mathcal{R}(T)$. Therefore, $\mathcal{C}$ is a clopen subbase for $\mathcal{R}(T)$ by Lemma \ref{descricao}. 
\end{proof}

Playing on the special basis described by \myref{especial}, we can explicit a winning strategy for Player $\mathrm{II}$:

\begin{prop}\label{estrategia}
    Let $T$ be a special tree. Then, Player $\mathrm{II}$ has a winning (stationary) strategy for the end game in $\mathcal{R}(T)$ played in the special basis generated by $\{[t,\emptyset]\}_{t\in T}$.
\end{prop}
\begin{proof}
    For every $t\in T$ we will denote by $\hat{t}$ the node $$\hat{t} = \max \{s\in T : s \text{ is a limit node and }s \leq t\}.$$ In addition, if $t$ is a limit node, we will write $\{t_i\}_{i < \omega}$ for a fixed cofinal sequence in $\mathring{\lceil t\rceil}$. Let $\mathcal{B}$ denote the special basis generated by $\{[t,\emptyset]\}_{t\in T}$, as in the expression \myref{especial}.

    Within this notation, we are ready to explicit a winning (stationary) strategy for Player $\mathrm{II}$ in $\mathrm{End}_{\mathcal{B}}$. To this aim, fix a basic open set $V = [t,F] \in  \mathcal{B}$. By definition of $[t,F]$ in \myref{intrinseco}, we can assume that $t < m$ for every $m\in F$. Then, consider the following three families of open sets:

    \begin{itemize}
        \item For every $m\in F$, define the index $i(m) = \min\{i \in\mathbb{N} : t < \hat{m}_i\}$. For every successor $s$ of $t$, the basic open set $[s, \Lambda_s^1(F)]$ will be referred as a \textit{type $1$} set, where $\Lambda_s^1(F) = \{\hat{m}_{i(m)} : m\in F, \hat{m}_{i(m)}> s\}\cup F$;
        \item For every $m\in F$, the open set $[\hat{m}_{i(m)}, \Lambda_m^2(F)]$ will be referred as a \textit{type $2$} set, where $$\Lambda_m^2(F) = (\{\hat{n}_{i(n)}: n \in F\}\cup \{\hat{n} : n \in F\}\cup F)\cap \{t\in T : t > \hat{m}_{i(m)}\};$$
        \item Finally, a \textit{type $3$} set is an open set of the form $[\hat{m}, \Lambda_m^3(F)]$, where $$\Lambda_m^3(F) = (\{\hat{n}_{i(n)}: n \in F\}\cup F)\cap \{t\in T : t > \hat{m}\}$$ for each $m\in F$.
    \end{itemize}

    Its is easily verified that the family $$\{[s,\Lambda_s^1(F)]: s\text{ is a successor of }t\}\cup \{[\hat{m}_{i(m)},\Lambda_m^2(F)]: m\in F \}\cup \{[\hat{m}, \Lambda_m^3(F)]: m \in F\}$$ covers $V = [t,F]$ with pairwise disjoint open sets. Hence, we will consider this open cover as the answer $\psi(V)$ of Player $\mathrm{II}$ when Player $\mathrm{I}$ declares $V$ in $\mathrm{End}_{\mathcal{B}}$. In order to see that $\psi$ is a winning strategy, let $$\langle V_0, \psi(V_0), V_1,\psi(V_1), V_2,\psi(V_2), \dots \rangle$$ be a match in which Player $\mathrm{II}$ follows the instructions of $\psi$. In particular, for each $n \in\mathbb{N}$, Player $\mathrm{I}$ chooses $V_{n+1}$ as an open set contained in $[t_n, F_n]\in \psi(V_n)$, for some $t_n \in T$ and some finite $F\subset T$. As in Lemma \ref{52}, $\{t_n\}_{n \in\mathbb{N}}$ is a totally ordered subset of $T$, because $[t_{n+1},F_{n+1}]\subset [t_n,F_n]$ for every $n \in\mathbb{N}$. Hence, there is $r$ a high-ray of $T$ in which $\{t_n\}_{n < \omega}$ is a cofinal sequence.

    In order to verify that Player $\mathrm{II}$ wins the match, we will prove that $\displaystyle \bigcap_{n\in\mathbb{N}}V_n$ can be written uniquely as a union of a high-ray and an open set not containing it. To that aim, we first observe that $\displaystyle \bigcap_{n\in\mathbb{N}}V_n = \displaystyle \bigcap_{n \in\mathbb{N}}[t_n,F_n]$, because $V_{n+1}\subset [t_{n},F_{n}]\subset V_n$ for every $n \in\mathbb{N}$. Moreover, there is no open set containing $r$ that is contained in $\displaystyle \bigcap_{n \in\mathbb{N}}[t_n,F_n]$, due to the fact that $\{t_n\}_{n \in\mathbb{N}}$ is cofinal in $r$. Hence, it is enough to prove that $ \displaystyle \bigcap_{n \in\mathbb{N}}[t_n,F_n]\setminus \{r\}$ is an open set in $\mathcal{R}(T)$.

    Indeed, fix a high-ray $r'\in \displaystyle \bigcap_{n \in\mathbb{N}}[t_n,F_n]\setminus \{r\}$. Note that $r'$ strictly contains $r$, once more because $\{t_n\}_{n \in\mathbb{N}}$ is cofinal in $r$. If $s$ denotes the top of $r$ belonging to $r'$, we finish the proof by concluding that there is no $m > s$ with $m\in \displaystyle \bigcup_{n\in\mathbb{N}}F_n$. In particular, we must have $[s,\emptyset]\subset \displaystyle \bigcap_{n \in\mathbb{N}}[t_n,F_n]$. For a contradiction, suppose that there is such a node $m > s$. If we choose $m$ with minimum height satisfying that property, we must have $m\in F_n$ for every $n$ bigger than some fixed $n_0 \in\mathbb{N}$. In particular, for every $n \geq n_0$, the open set $[t_n,F_n]$ is not of type $1$. Neither it is of type $3$: otherwise, $\hat{t_n} = \hat{m}$, so that the chain $\{t\in T : t_n \leq t \leq m\}$ is finite, contradicting the fact that $m > s$. Hence, the open set $[t_n,F_n]$ is of type $2$ for every $n \geq n_0$. In particular, by the minimality of $m$, we must have $ m = \hat{m}$ and the sequence $\{t_n\}_{n \geq n_0}$ must be cofinal in $\mathring{\lceil m \rceil}$. Therefore, $m = s$, contradicting the choice of $m$.   
 \end{proof}

 From now on in this section, our aim is to prove a converse statement for Proposition~\ref{estrategia}. In other words, if $X$ is a Hausdorff topological space in which $\mathrm{II}\uparrow \mathrm{End}_{\mathcal{B}}$ for some special basis $\mathcal{B}$, we will conclude that $X$ is the high-ray space of some special order tree. To that aim, we will study suitable moves for Player $\mathrm{II}$ against a winning strategy for Player $\mathrm{II}$.

 Before that, we will explicit a convenient method to partition a basic open set from a basis $\mathcal{B}$ generated by a special clopen subbase $\mathcal{C} = \{U_{\alpha}\}_{\alpha\in \Lambda}$. Considering the expression \myref{descricaobase}, this open set is written as $[U,F] = \displaystyle U\setminus \bigcup_{\alpha\in F}U_{\alpha}$, for some $U\in \mathcal{C}$ and some finite $F\subset \Lambda$. In particular, we can always assume that $U_{\alpha}\cap U_{\beta} = \emptyset$ for every pair $\alpha,\beta\in F$, because $\mathcal{C}$ is a nested family. By the same reason, $U_{\alpha}$ is a proper subset of $U$ for every $\alpha \in F$, unless $[U,F]=\emptyset$. When comparing basic open sets, the following observation is often implicitly applied:

 \begin{lemma}\label{decrescente}
     Let $[U,F], [V,G]\in \mathcal{B}$ be basic open sets such that $[V,G]\subset [U,F]$. Then, $[V,G] = [W,K]$ for some subbasic clopen set $W\subseteq U$ and some finite $K\subset \Lambda$.
 \end{lemma}
 \begin{proof}
     Since $[V,G]\subset [U,F]$, we have $V\cap U \neq \emptyset$, unless $[V,G] = \emptyset$. If $V\subseteq U$, the result is immediate. If not, then $U\subseteq V$, because $\mathcal{C}$ is a nested family. In this case, we claim that $[V,G] = [U,F\cup G]$, finishing the proof. In fact, if $x \in [V,G]$, then $x \in [U,F]$ and $x \notin U_{\alpha}$ for any $\alpha \in G$. Hence, $x \in [U,F\cup G]$. Conversely, if $x\in [U,F\cup G]$, then, in particular, $x \notin U_{\alpha}$ for every $x\in V$ and any $\alpha \in G$, because $U\subseteq V$. In other words, $x\in [V,G]$. 
 \end{proof}

 On the other hand, for each $U\in \mathcal{C}$ that is a proper clopen set of $X$, the family $\mathring{\lceil U\rceil} := \{V\in \mathcal{C}: U\subsetneq V\subsetneq X\}$ is either empty or well ordered by $\supseteq$, since $\mathcal{C}$ is nested and noetherian. In the latter case, $\mathring{\lceil U\rceil}$ is countable, since $\mathcal{C}$ is $\sigma-$disjoint. Then, we will denote by $\rho(U) \subset \mathring{\lceil U\rceil}$ a fixed cofinal sequence in $\mathring{\lceil U\rceil}$, for some $\kappa_u \in \omega+1$. In addition, the following observation is useful:

 \begin{lemma}\label{doiscasos}
     For a basic open set $[U,F]\in \mathcal{B}$, precisely one of the items below holds:
     \begin{itemize}
         \item[$i)$] Either, for every $x \in [U,F]$, there is a clopen set $U_x \in \mathcal{C}$ such that $x\in U_x\subsetneq U$;
         \item[$ii)$] Or there is a unique $x \in [U,F]$ such that $U$ is the smallest clopen set from $\mathcal{C}$ containing $x$, namely, there is no $V\in \mathcal{C}$ such that $x\in V\subsetneq U$.  
     \end{itemize}
 \end{lemma}
 \begin{proof}
     Assume that there are distinct $x,y\in [U,F]$ such that $U$ is the smallest clopen set from $\mathcal{C}$ containing them. Since $X$ is a Hausdorff topological space, there are $U_x,U_y \in \mathcal{U}$ and finite sets $F_x,F_y\subset \Lambda$ such that $x\in [U_x,F_x]\subset [U,F]$ and $y\in [U_y,F_y]\subset [U,F]$, but $[U_x,F_x]\cap [U_y,F_y] = \emptyset$. By the minimality of $U$, we must have $U\subseteq U_x$ and $U\subseteq U_y$, so that $x,y\in U\subseteq U_x \cap U_y$. We then can assume that $U_x\subseteq U_y$, because $\mathcal{C}$ is a nested family. Therefore, there is $i\in F_y$ such that $x \in U_i$, since $x \in U_y$ but $[U_x,F_x]\cap [U_y,F_y] = \emptyset$. Hence, $U\subseteq U_i$, again by the minimality of $U$. This, however, contradicts the fact that $y\in [U_y,F_y] = \displaystyle U_y \setminus \bigcup_{j\in F_y}U_j$.
 \end{proof}

Based on the cases displayed by Lemma \ref{doiscasos}, precisely one of the four items below describe a partition for $[U,F]$ into basic open sets from $\mathcal{B}$. For further reference, such an open cover will be denoted by $\mathcal{K}'[U,F]$. In addition, although a basic open set $V\in \mathcal{K}'[U,F]$ might admit more than one representation of the form $V = [U',F']$, the following procedure will \textit{fix} a clopen set $U'\in \mathcal{C}$ and a finite $F'\subset \Lambda$ such that $V=[U',F']$. Moreover, this choice will be done so that $U'\subseteq U$:

\begin{enumerate}
    \item \label{item1} Assume first that item $i)$ from the above lemma is verified. Fix, for every $x\in [U,F]$, a clopen set $U_x\in \mathcal{C}$ such that $x\in U_x\subsetneq U$. If, for some $A\subset [U,F]$, the subfamily $\{U_x: x \in A\}$ is a chain regarding $\subseteq$, there is $\tilde{x}\in A$ such that $U_{\tilde{x}}$ is a maximum element in  $\{U_x: x \in A\}$, because $\mathcal{C}$ is noetherian. Hence, by Zorn's Lemma, the set $$\tilde{A} = \{\tilde{x}\in [U,F]: U_{\tilde{x}} \text{ is }\subseteq-\text{maximal in }\{U_x\}_{x\in [U,F]}\}$$ is well-defined. Actually, this argument shows that, for every $x \in [U,F]$, there is $\tilde{x}\in \tilde{A}$ so that $x \in U_{\tilde{x}}$. Moreover, given $\tilde{x},\tilde{y} \in \tilde{A}$, then $U_{\tilde{x}}\cap U_{\tilde{y}} \neq \emptyset$ or $U_{\tilde{x}} = U_{\tilde{y}}$ by the $\subseteq-$maximality of both $U_{\tilde{x}}$ and $U_{\tilde{y}}$, since $\mathcal{C}$ is nested. Hence, $\{U_{\tilde{x}}\}_{\tilde{x}\in \tilde{A}}$ is a disjoint family. In particular, for every $\alpha\in F$, there is a unique $\tilde{x}_{\alpha}\in \tilde{A}$ such that $U_{\alpha}\subset U_{\tilde{x}_{\alpha}}$. Then, let $U_{\alpha}'$ be the $\subseteq-$maximum set from $\{V \in \rho(U_{\alpha}) : U_{\alpha}\subseteq V \subsetneq U_{\tilde{x}_{\alpha}}\}$. Within this notation, $$\mathcal{K}'[U,F] = \left\{U_{\tilde{x}}\setminus \bigcup_{\alpha\in F}U_{\alpha}' : \tilde{x}\in \tilde{A}\right\}\cup\left\{U_{\alpha}'\setminus  \bigcup_{\substack{\beta\in F \\ U_{\alpha}'\not\subset U_{\beta}'}}U_{\beta}': \alpha \in F\right\}$$ is a disjoint open cover for $[U,F]$ whose elements belong to the basis $\mathcal{B}$; 
    \item \label{item2} For this and the next two cases, assume that item $ii)$ of Lemma \ref{doiscasos} holds. Moreover, for each $\alpha \in F$, denote by $U_{\alpha}'$ the $\subseteq-$maximum set from $\{W \in \rho(U_{\alpha}) : U_{\alpha}\subsetneq W \subsetneq U\}$, if it exists. Otherwise, denote $U_{\alpha}' = U_{\alpha}$. Define the set $F' = \{\alpha\in F : U_{\alpha}'\neq U_{\alpha}\}$. If $F'\neq \emptyset$, then $$\mathcal{K}'[U,F] = \left\{U\setminus \bigcup_{\alpha\in F}U_{\alpha}'\right\}\cup \left\{U_{\alpha}'\setminus  \bigcup_{\substack{\beta\in F \\ U_{\alpha}'\not\subset U_{\beta}'}}U_{\beta}' : i \in F'\right\}$$
    is a disjoint open cover for $[U,F]$ whose elements belong to the basis $\mathcal{B}$; 
    \item \label{item3} Following the notation from the above item, suppose now that $F' = \emptyset$ and that $[U,F] = \{x\}$ is singleton. In this case, we set the trivial open cover $\mathcal{K}'[U,F] = \{[U,F]\} = \{\{x\}\}$;
    \item \label{item4} Finally, let $x$ be the unique point from $[U,F]$ such that $U$ is the smallest clopen set of $\mathcal{C}$ containing $x$. Following the notation from the previous items, suppose that $F' = \emptyset$ and fix $y \in [U,F]\setminus \{x\}$. Then, there is $U_y\subsetneq U$ a subbasic set containing $y$. Since $y\notin U_{\alpha}$ for any $\alpha\in F$, it follows that $U_y \cap U_{\alpha} = \emptyset$, because $F' = \emptyset$ and $\rho(U_{\alpha})$ is a cofinal sequence in $\mathring{\lceil U_{\alpha}\rceil }$. In this case, $$\mathcal{K}'[U,F] = \left\{U_y,[U,F]\setminus U_y\right\}$$ is the claimed open cover for $[U,F]$.
\end{enumerate}

Regarding these definitions, we can study convenient matches between Player $\mathrm{I}$ and a fixed winning strategy $\psi$ for Player $\mathrm{II}$. First, for a basic open set $[U,F]\in \mathcal{B}$, we remark that any other element from $[V,G]\in \psi([U,F])$ is contained in $[U,F]$. Hence, relying on Lemma~\ref{decrescente}, we can assume that $V\subseteq U$.

Now, suppose that Player $\mathrm{I}$ starts the game $\mathrm{End}_{\mathcal{B}}$ by declaring a basic open set $[U_0,F_0]$. At the $n-$th round of the game, suppose that Player $\mathrm{II}$ declared a disjoint open cover $\mathcal{U}_n\subset \mathcal{B}$ for the basic open set $[U_n,F_n]$ just played by Player $\mathrm{I}$. We write $\psi([U_n,F_n]) = \mathcal{U}_n$, assuming that the answer of Player $\mathrm{II}$ follows the given winning strategy. Then, we consider the case in which Player $\mathrm{I}$ starts the $(n+1)-$th round by choosing any basic open set $[U_{n+1},F_{n+1}]\in \displaystyle \bigcup_{[U,F]\in \mathcal{u}_n}\mathcal{K}'[U,F]$. In particular, $[U_{n+1},F_{n+1}]\in \mathcal{K}'[\tilde{U_n},\tilde{F_n}]$ for a unique basic open set $[\tilde{V_n},\tilde{F_n}]\in \mathcal{U}_n$.  Therefore, we described a match $$\langle [U_0,F_0],\mathcal{U}_0, [U_1,F_1], \mathcal{U}_1,[U_2,F_2], \mathcal{U}_2,\dots \rangle.$$ Since $\psi$ is a winning strategy for Player $\mathrm{II}$, we therefore can uniquely write $$\bigcap_{n\in\mathbb{N}}[U_{n},F_{n}] = \{x\} \cup A,$$ where $A\subset X$ is open and $x \in X\setminus A$. Moreover, $[\tilde{U_{n+1}},\tilde{F_{n+1}}]\subset [U_{n+1},F_{n+1}]\subset [\tilde{U_{n}},\tilde{F_n}]$ for every $n \in\mathbb{N}$. Hence, although $A$ might not be a basic open set from $\mathcal{B}$, we will ensure below that $\{A\} \cup \mathcal{C}$ is a nested family. As a consequence, one can show that $A$ is a basic open set or a disjoint union of subbasic sets:

\begin{prop}\label{tricotomiaA}
    Given $U\in \mathcal{C}$, then $A\cap U = \emptyset$, $U\subseteq A$ or $A\subseteq U$.
\end{prop}
\begin{proof}
    Suppose that $A\cap U \neq \emptyset$. In particular, $U\cap U_n \neq \emptyset$ for every $n \in \mathbb{N}$. If $U_n\subset U$ for some $n \in\mathbb{N}$, we are done, because $A\subset U_n$. Therefore, since $\mathcal{C}$ is a nested family, we can assume that $U\subsetneq U_n$ for each $n \in\mathbb{N}$. However, for a contradiction, suppose that $U\not\subset [U_n,F_n]$ for some $n \in\mathbb{N}$. Hence, $U\cap U_{\alpha} \neq \emptyset$ for some $\alpha\in F_n$, so that $U_{\alpha}\subsetneq U$ (because $\mathcal{C}$ is nested and $U\cap A \neq \emptyset$). This verifies that the set $$\left\{U_{\beta} \in \mathcal{C}: \beta\in \displaystyle \bigcup_{m\in\mathbb{N}}U_m, U_{\alpha}\subseteq U_{\beta} \subsetneq U \right\}$$ is not empty and, since $\mathcal{C}$ is noetherian, there is $U_{\beta}$ a $\subseteq-$maximal element. Hence, $\beta \in \displaystyle \bigcup_{m \geq n_0}U_m$ for some big enough $n_0\in\mathbb{N}$, since $[U_{m+1},F_{m+1}]\subset [U_m,F_m]$ for every $m \in\mathbb{N}$. Besides that, if $m > n_0$, then $\mathcal{K}'[\tilde{U_m},\tilde{F_m}]$ is not constructed as in items \ref{item3} and \ref{item4}, because, once $U_{\beta}\subsetneq U\subsetneq \tilde{U_m}$, the set $\{V \in \rho(\tilde{U_m}): U_{\beta}\subsetneq V \subsetneq \tilde{U_m}\}$ is not empty.

    Therefore, $\mathcal{K}'[\tilde{U_m},\tilde{F_m}]$ is defined according to items \ref{item1} and \ref{item2}. In particular, $U_{m+1}\in \rho(U_{\beta})$ for every $m > n_0$, because $U_{\beta}\subset U \subset  \displaystyle \bigcap_{n \in\mathbb{N}}U_n$. This means that $\{U_m\}_{m > n_0+1}$ is an infinite subsequence of $\rho(U_{\beta})$, being also cofinal (regarding $\supseteq$) in $\mathring{\lceil U_{\beta}\rceil}$. Since $U_{\beta}\subsetneq U$, we must have $U_m \subsetneq U$ for some $m > n_0+1$, which is a contradiction.

    Hence, we verified that $U\subset \displaystyle \bigcap_{n \in\mathbb{N}}U_n$. If $A\subset U$, we are done. If not, we will show that $x\notin U$, which concludes the inclusion $U\subset A$. To that aim, fix $y\in A\setminus U$ and suppose that $x \in U$. Since $U$ is a clopen set, $(A\setminus U)\setminus \{y\}$ is open, because so is $A\setminus U$. Then, we can write $\{x\}\cup A$ as $\{y\} \cup U\cup ((A\setminus U)\setminus \{y\})$, contradicting the unique representation of $\displaystyle \bigcap_{n\in\mathbb{N}}[U_{n+1},F_{n+1}]$ as the disjoint union of a point and an open set.  
\end{proof}
\begin{cor}\label{descricaoA}
    We can write $A = \displaystyle \bigcup_{U\in \mathcal{K}[A]} U$ for some disjoint family $\mathcal{K}[A]\subset \mathcal{C}$ or  $A$ is a basic open set. If this latter case, we write $\mathcal{K}[A]:=A$
\end{cor}
\begin{proof}
    First, suppose that there is $y\in A$ such that $A\subset U$ for every $U\in \mathcal{C}$ containing $y$. Since $A$ is open, there is $[U,F]\in\mathcal{B}$ such that $y\in [U,F]\subset A$. Hence, we must have $A\subset U$. Moreover, for every $\alpha\in F$ such that $A\cap U_{\alpha}\neq \emptyset$, we must have $U_{\alpha}\subseteq A$ by the above Proposition, because $y\notin U_{\alpha}$. Therefore, by considering the family $F' = \{\beta \in F : A\cap U_{\beta} = \emptyset\}$, we conclude that $A = [U,F']$ is a basic open set.

    Now, suppose that, for every $y\in A$, there is $U_y \in \mathcal{C}$ satisfying $y\in U_y \subset A$. Since $\mathcal{C}$ is nested and noetherian, we can apply Zorn's Lemma to define the set $$\tilde{A} = \{\tilde{y} \in A: U_{\tilde{y}}\text{ is }\subseteq-\text{maximal in }\{U_{y}\}_{y\in A} \}$$ and verify that $\mathcal{K}[A]:=\{U_{\tilde{y}}: \tilde{y}\in \tilde{A}\}$ is a disjoint cover for $A$. Actually, this is precisely the same argument presented by the definition of $\mathcal{K}'[U,F]$ as in item \ref{item1}.
\end{proof}

The description of the open set $A$ as in Corollary \ref{descricaoA} suggests that the end game $\mathrm{End}_{\mathcal{B}}$ could be played again, with Player $\mathrm{I}$ starting the match by declaring $A$ (if it is a basic open set) or some clopen set from $\mathcal{K}[A]$. Hence, by iteratively playing $\mathrm{End}_{\mathcal{B}}$ against the winning strategy $\psi$ of Player $\mathrm{II}$, we can describe a tree $T_{\mathcal{C}}$ by induction as follows:  

\begin{itemize}
    \item  Without loss of generality, we can assume that $X\in \mathcal{C}$. Then, we set $X = [X,\emptyset]$ as the root of $T_{\mathcal{C}}$;
    \item Every node of $T_{\mathcal{C}}$ is some basic open set $[U,F]$ from $\mathcal{B}$, possibly even a subbasic one when considering $F = \emptyset$. Moreover, $T_{\mathcal{C}}$ is ordered by the inverse inclusion $\supseteq$. In this case, we consider the (disjoint) open cover $\mathcal{W} = \psi([U,F])$. Then, the successors of $[U,F]$ in $T_{\mathcal{C}}$ are precisely the open sets from $$\mathcal{K}[U,F] : =\displaystyle \bigcup_{[\tilde{U},\tilde{F}]\in \mathcal{W}}\mathcal{K}'[\tilde{U},\tilde{F}],$$ described by the items \ref{item1}-\ref{item4};
    \item For a fixed high-ray $R$ of $T_{\mathcal{C}}$, we consider $\{[U_n,F_n]\}_{n\in \mathbb{N}}$ a cofinal sequence in $R$. We can see these open sets as answers for Player $\mathrm{I}$ in the match $$\langle [U_0,F_0],\psi([U_0,F_0]), [U_1,F_1], \psi([U_1,F_1]), [U_2,F_2], \psi([U_2,F_2]),\dots \rangle$$ 
    Since $\psi$ is a winning strategy for Player $\mathrm{II}$, we can write uniquely $$\displaystyle \bigcap_{n =0}^{\infty}[U_n,F_n] = \{x\} \cup A$$ for some point $x\in X$ and some open set $A\subset X$ not containing $x$. Then, we define $\mathcal{K}[A]$ as the set of tops of the high-ray $R$, where $\mathcal{K}[A]$ is given by Corollary \ref{descricaoA}. If $\{[V_n,G_n]\}_{n\in\mathbb{N}}$ is another cofinal sequence in $R$, then $\displaystyle \bigcap_{n \in \mathbb{N}}[U_n,F_n] = \displaystyle \bigcap_{n\in\mathbb{N}}[V_n,G_n]$, because $[U_{n+1},F_{n+1}]\subset [U_n,F_n]$ and $[V_{n+1},G_{n+1}]\subset [V_n,G_n]$ for every $n\in \mathbb{N}$. In other words, the description of $A$ does not depend on the choice of $\{[U_n,F_n]\}_{n\in \mathbb{N}}$. Hence, this procedure defines the limit nodes of $T_{\mathcal{C}}$. 
\end{itemize}

 Relying on the fact that $\mathcal{C}$ is $\sigma-$disjoint, we will now argue how $T_{\mathcal{C}}$ is a special tree. To this aim, we observe that, for every $U\in \mathcal{C}$, it is enough to construct a partition $\{\mathcal{A}_U^n\}_{n\in\mathbb{N}}$ of $$T_U = \{[U,F]: [U,F]\in T_{\mathcal{C}} \text{ for some finite }F\subset \Lambda\}$$ into countably many disjoint families. This because, if $\{\mathcal{C}_k\}_{k\in \mathbb{N}}$ is a partition of $\mathcal{C}$ into antichains, then $\{T_k^n:n,k\in\mathbb{N}\}$ is a partition of $T$ into antichains, where $T_k^n = \displaystyle \bigcup_{U\in \mathcal{C}_k}\mathcal{A}_U^n$ for every $n,k\in\mathbb{N}$. Indeed, the description of $\{\mathcal{A}_U^n\}_{n\in\mathbb{N}}$ is the core of the following lemma:
 
\begin{lemma}
    For each $U\in \mathcal{C}$, there is $\{\mathcal{A}_U^n\}_{n\in\mathbb{N}}$ a cover of $T_U$ by antichains. In particular, $T_{\mathcal{C}}$ is a special tree.
\end{lemma}
\begin{proof}
     Consider $\mathcal{A}_U^0$ the set of $\subseteq-$maximal elements from $T_U$, or, equivalently, the set of minimal elements of $T_U$ is the tree order of $T_{\mathcal{C}}$. Clearly, $\mathcal{A}_U^0$ is an antichain. For each $V\in \mathcal{A}_U^0$, it is enough to conclude that the set $\downarrow V = \{V' \in T_U : V'\subsetneq V\}$ is countable, written as $\{V_i\}_{i < \kappa_V}$ for some cardinal $\kappa_V \leq \omega$. Then, for $n \geq 1$, the claimed antichain $\mathcal{A}_U^n$ can be given by $\mathcal{A}_U^n = \{V_n : V \in \mathcal{A}_U^0 \text{ and }\kappa_V \geq n\}$.

     In fact, a basic open set  $V\in \mathcal{A}_U^0$ can be written as $[U,F]$ for some finite $F\subset \Lambda$. First, suppose that, for every $x\in V$, there is a clopen set $U_x\in \mathcal{C}$ such that $x\in U_x \subsetneq U$. For a basic open set $[U_0,F_0]\in \psi([U,F])$, we have either $U_0\subsetneq U$ or $U_0=U$. In the former case, every element from $\mathcal{K}'[U_0,F_0]$ can be written as $[U_1,F_1]$ for some clopen set $U_1\subseteq U_0 \subsetneq U$. In the latter, $\mathcal{K}'[U_0,F_0]$ is defined according to item \ref{item1}, by the assumption over $U$. Then, every element from $\mathcal{K}'[U_0,F_0]$ has the form $[U_1,F_1]$ for some clopen set $U_1\subsetneq U_0 = U$. In both cases, $U_1\subsetneq U$ for each $[U_1,F_1]\in \mathcal{K}[U,F]$. As a consequence, for every $[U_2,F_2]\in T_{\mathcal{C}}$ such that $[U_2,F_2]\subsetneq [U,F]$, we must have $U_2\subsetneq U$. Hence, $\downarrow V = \emptyset$ in this case.

     Now, according to the dichotomy of Lemma \ref{doiscasos}, we assume the existence of a unique $x\in [U,F]$ such that $U$ is the smallest clopen set from $\mathcal{C}$ containing $x$. Then, $\mathcal{K}'[\tilde{U},\tilde{F}]$ is constructed following items \ref{item2}-\ref{item4}, if $[\tilde{U},\tilde{F}]$ denotes the open set from $\psi([U,F])$ containing $x$. In any of the three criteria, there is a unique $V'\in \mathcal{K}[U,F]$ of the form $V' = [U,F_1]$ for some finite $F_1\subset \Lambda$, i.e., $T_U\cap \mathcal{K}[U,F] = \{V'\}$. Indeed, $V'$ is the open set from $\mathcal{K}[U,F]$ containing $x$. By induction, suppose that we have defined $F_1,F_2,\dots,F_n \subset \Lambda$ such that $T_U\cap \mathcal{K}[U,F_i] = \{[U,F_{i+1}]\}$ for every $1\leq i < n$. Moreover, assume that $x\in [U,F_i]$ for every $1\leq i \leq n$. As before, item $ii)$ of Lemma \ref{doiscasos} holds, so that the open set from $\mathcal{K}[U,F_n]$ containing $x$ is the unique that can be written as $[U,F_{n+1}]$ for some finite $F_{n+1}\subset \Lambda$. More precisely, $T_U\cap \mathcal{K}[U,F_n] = \{F_{n+1}\}$, because, if $[\tilde{U},\tilde{F}]$ denotes the open set from $\psi([U,F])$ containing $x$, the partition $\mathcal{K}'[\tilde{U},\tilde{F}]$ was again constructed according to items \ref{item2}-\ref{item4}.

     At the end of this recursive process, $\{[U,F_n]\}_{n\in\mathbb{N}}$ defines a cofinal sequence in a high-ray $R$. Alternatively, we can see $\{[U,F_n]\}_{n\in\mathbb{N}}$ as the moves of Player $I$ in the match $$\langle [U,F_0],\psi([U,F_0]), [U,F_1],\psi([U,F_1]), [U,F_2],\psi([U,F_2]),\dots \rangle$$ For instance, suppose that there is an open neighborhood of $x$ in $\displaystyle \bigcap_{n=0}^{\infty}[U,F_n]$ containing at least two elements. Since $U$ is the smallest clopen set from $\mathcal{C}$ containing $x$, we can represent this open neighborhood as $[U,F_{\infty}]$ for some finite $F_{\infty}\subset \Lambda$. Assuming that $F_{\infty}$ is $\subseteq-$minimal with this property, for every $\alpha \in F_{\infty}$ there are $n_{\alpha}\in \mathbb{N}$ and $\eta_{\alpha}\in F_{n_{\alpha}}$ such that $U_{\eta_{\alpha}}\subseteq U_{\alpha}$. If $U_{\eta_{\alpha}}$ is a proper subset of $U_{\alpha}$, then $\{W\in \rho(U_{\eta_{\alpha}}): U_{\eta_{\alpha}}\subsetneq W\subsetneq U\}$ is not empty, so that $\mathcal{K}'[\tilde{U},\tilde{F}]$ is constructed as in item \ref{item2}. Again, $[\tilde{U},\tilde{F}]\in \psi([U,F_{n_{\alpha}}])$ denotes the open set containing $[U,F_{n_{\alpha}+1}]$. Hence, there is $\eta_{\alpha}'\in F_{n_{\alpha}+1}$ such that $ U_{\eta_{\alpha}}\subsetneq U_{\eta_{\alpha}'}\subseteq U_{\alpha}$. Since there are no infinite $\subseteq-$increasing chains in $\mathcal{C}$, we can actually choose, for each $\alpha\in F_{\infty}$, a index $n_{\alpha}\in \mathbb{N}$ so that $\alpha \in F_{n_{\alpha}}$.

     Therefore, if $n =\max\{n_{\alpha} : \alpha\in F_{\infty}\}$, then $\mathcal{K}'[\tilde{U},\tilde{F}]$ is described by item \ref{item4}, where  $[\tilde{U},\tilde{F}]\in \psi([U,F_n])$ is the open set containing $[U,F_{n+1}]$. Then, $F_{n+1} = F_n\cup \{\beta\}$ for some index $\beta\in \Lambda \setminus F_n$ such that $U_\beta\cap U_{\beta'} = \emptyset$ for every $\beta'\in F_n$. On the other hand, $F_n = F_{\infty}$ by the choice of $n$ and the $\subseteq-$minimality of $F_{\infty}$. Hence, we contradict the fact that $[U,F_{\infty}]\subset [U,F_{n+1}]$.

     In other words, we proved that, when writing $\displaystyle \bigcap_{n=0}^{\infty}[U,F_n]$ as a disjoint union of a point and an open set $A$, this point must be $x$. By Corollary \ref{descricaoA}, $A$ has either the form $[W,H]\in \mathcal{B}$ or it is a disjoint union of a family $\mathcal{K}[A]\subset \mathcal{C}$. In the former case, we note that $W\subsetneq U$: otherwise, $W = U$ and $x \in U_{\alpha}$ for some $\alpha\in H$, contradicting the minimality of $U$ as a clopen set from $\mathcal{C}$ containing $x$. In the latter, $U \notin \mathcal{K}[A]$ because $x\notin A$. In both cases, if $[U',F']\in T_{\mathcal{C}}$ is contained in some top of $R$, then $U'\subsetneq U$. This proves that $\downarrow V = \{[U,F_n]: n \geq 1\}$.   
\end{proof}

In particular, the last paragraph of the above proof verifies the following useful property:

\begin{cor}\label{mesmabase}
    Consider an infinite chain in $T_{\mathcal{C}}$ of the form $\{[U,F_n]\}_{n\in\mathbb{N}}$, for some $U\in \mathcal{C}$. Then, let $$\bigcap_{n=0}^{\infty}[U,F_n] = \{x\} \cup A$$ be the unique description of $\displaystyle \bigcap_{n=0}^{\infty}[U,F_n]$ as the disjoint union of a point and an open set. Then, as in item $ii)$ of Lemma \ref{doiscasos}, $U$ is the smallest clopen set of $\mathcal{C}$ containing $x$.
\end{cor}

To summarize, $T_{\mathcal{C}}$ displays suitable basic open sets from $\mathcal{B}$ via a tree ordered by $\supseteq$. In particular, any chain in $T_{\mathcal{C}}$ of the form $[U_0,F_0]\subseteq [U_1,F_1]\subseteq [U_2,F_2]\subseteq \dots$ must stabilize, i.e., there must exist $[U_n,F_n]$ a $\subseteq-$maximal element for some $n \in \mathbb{N}$. Moreover, if $[U,F]$ and $[V,G]$ are incomparable in the tree order of $T_{\mathcal{C}}$, so that $[U,F]\not\subseteq [V,G]$ and  $[V,G]\not\subseteq [U,F]$, then $[U,F]\cap [V,G] = \emptyset$. This because $\bigcap \{[W, H] \in T_{\mathcal{C}} : [W,H]\subseteq [U,F] \text{ and }[W,H]\subseteq [V,G] \}$ is either a basic open set $I = [W',H']$ or a disjoint union of a point and an open set $I$. In both cases, $[U,F]$ and $[V,G]$ are subsets of distinct elements from a disjoint open cover for $I$. In particular, any antichain in the tree $T_{\mathcal{C}}$ is an antichain of $X$ as a topological space. In other words, the set $\psi(\mathcal{C}) := \{[U,F] \in \mathcal{B}: [U,F]\in T_{\mathcal{C}}\}$ of nodes of $T_{\mathcal{C}}$ is a noetherian, nested and $\sigma-$disjoint family. Moreover, its elements are also clopen sets, as observed by Lemma \ref{descricao}.

In particular, for every $x\in X$, the set $\mathcal{V}_x' = \{[U,F]\in \psi(\mathcal{C}): x \in [U,F]\}$ is a chain in  $\psi(\mathcal{C})$, because, once $\psi(\mathcal{C})$ is nested, its elements are pairwise comparable. Besides that, for every $[U,F]\in \mathcal{V}_x'$, there is a (unique) basic open set from $\mathcal{K}[U,F]$ containing $x$, since $\mathcal{K}[U,F]$ is a disjoint open cover for $[U,F]$. Therefore, $\mathcal{V}_x'$ describes a high-ray in $T_{\mathcal{C}}$, so that $\displaystyle \bigcap \mathcal{V}_x'$ is written uniquely as a union of a point and an open set $A$ not containing it. We observe that this point is precisely $x$: otherwise, $x$ belongs to some $[U,F]\in \mathcal{K}[A]$, contradicting the definition of $\mathcal{V}_x'$. Due to this fact, we write $A_x = A$ and state the following remark:  

\begin{prop}\label{baselocal}
    For every $x\in X$, fix $\{[U_n,F_n]\}_{n\in \mathbb{N}}$ a countable cofinal sequence in $\mathcal{V}_x'$ regarding the tree order of $T_{\mathcal{C}}$, whose existence follows from the fact that $T_{\mathcal{C}}$ is special. Then, the set $$\mathcal{V}_x = \left\{[U_n,F_n]\setminus \bigcup_{U\in\mathcal{F}} U : n\in\mathbb{N}, \mathcal{F}\subset \mathcal{K}[A_x]\text{ finite}\right\}$$ is a local basis for $x$. In particular, by Lemma \ref{descricao}, $\psi(\mathcal{C})$ is a special clopen subbase for $X$. 
\end{prop}
\begin{proof}
    Since $ \{[U_n,F_n]\}_{n\in \mathbb{N}}$ is cofinal in $\mathcal{V}_x'$, we have $ \displaystyle \bigcap_{n=0}^{\infty}[U_n,F_n] = \bigcap_{A\in \mathcal{V}_x'} A = \{x\}\cup A_x$. Then, let $[W,H]$ be a basic open set from $\mathcal{B}$ containing $x$. For every $\theta \in H$, we can assume that $U_{\theta}\cap [U_n,F_n] \neq \emptyset$ for every $n \in \mathbb{N}$: in this case, an open set $V\in \mathcal{V}_x$ such that $x \in V \subset [W,H\setminus \{\theta\}]$ also satisfies $x \in V \subset [W,H]$. In particular, $U_n \cap U_{\theta} \neq \emptyset$ for every $n \in \mathbb{N}$, so that $U_{\theta}\subsetneq U_n$ for every $n\in\mathbb{N}$ because $\mathcal{C}$ is nested and $x\in U_n\setminus U_{\theta}$. Analogously, for every $\theta \in H$ and every $\alpha \in \displaystyle \bigcup_{n \in \mathbb{N}}F_{n}$, we can assume that $U_{\theta}\cap U_{\alpha} = \emptyset$ or $U_{\alpha}\subsetneq U_{\theta}$. Otherwise, since $\mathcal{C}$ is nested, we would have $U_{\theta}\subseteq U_{\alpha}$ for some $\alpha \in F_n$ and some $n\in \mathbb{N}$. Actually, there would exist, for every $k \geq n$, an index $\alpha'\in F_{k}$ such that $U_{\theta}\subseteq U_{\alpha'}\subseteq U_{\alpha}$. In this case, if $V\in \mathcal{V}_x$ satisfies $x\in V\subset [W,H]$ and has the form $\displaystyle [U_m,F_m]\setminus \bigcup_{U\in\mathcal{F}} U$ for some $m\in \mathbb{N}$, then $V' = \displaystyle [U_{m+n+1}, F_{m+n+1}]\setminus \bigcup_{U\in\mathcal{F}} U$ is an open set from $\mathcal{V}_x$ such that $V'\subset [W,H\setminus \{\theta\}]$.

    Within the above considerations, suppose first that $U_{\alpha}\subsetneq U_{\theta}$ for some $\alpha \in \displaystyle \bigcup_{n\in \mathbb{N}}F_n$ and some $\theta \in H$. We can choose $\alpha$ so that $U_{\alpha}$ is $\subseteq-$maximal satisfying $U_{\alpha}\subsetneq U_{\theta}$, because $\mathcal{C}$ is noetherian. If $n \in \mathbb{N}$ is an index such that $\alpha \in F_n$, then $U_{\alpha}\subsetneq U_{\theta}\subsetneq U_n$. However, if $[\tilde{U}_n,\tilde{F}_n]\in \psi([U_n,F_n])$ is the open set containing $[U_{n+1},F_{n+1}]$, then $\mathcal{K}'[\tilde{U}_n,\tilde{F}_n]$ is constructed following items \ref{item1} or \ref{item2}. In both cases, there is $\alpha'\in F_{n+1}$ such that $U_{\alpha}\subsetneq U_{\alpha'}$. By the choice of $\alpha$, we must have $U_{\theta}\subseteq U_{\alpha'}$, a situation that was discarded by the previous paragraph. Therefore, from now on, we assume that $U_{\theta}\cap U_{\alpha} = \emptyset$ for every $\theta \in H$ and every $\alpha \in \displaystyle \bigcup_{n\in\mathbb{N}}F_n$.

    On the other hand, suppose that $W\subsetneq U_n$ for every $n\in\mathbb{N}$. Then, there must exist $n_0 \in \mathbb{N}$ such that $[W,H]\cap U_{\alpha} \neq \emptyset $ for some $\alpha\in F_{n_0}$: otherwise, $\displaystyle x\in [W,F]\subset \bigcap_{A\in \mathcal{V}_x'} A$, contradicting the uniqueness of the representation $\displaystyle \bigcap_{A\in \mathcal{V}_x'} A = \{x\}\cup A_x$. Then, $U_{\alpha}\subsetneq W$, because $\mathcal{C}$ is a nested family and $x\in W\setminus U_{\alpha}$. Moreover, since $\mathcal{C}$ is noetherian, we assume that $U_{\alpha}$ is a $\subseteq-$maximal element from $\displaystyle \bigcup_{n=1}^{\infty}F_n$ such that $U_{\alpha}\subsetneq W$. Therefore, one of the following cases is verified, but both lead to contradictions:

    \begin{itemize}
        \item Suppose for instance that $U_{n+1}\subsetneq U_n$ for each $n\in \mathbb{N}$. By the $\subseteq-$maximality of $U_{\alpha}$, we have $\alpha \in F_n$ for every $n \geq n_0$. Regarding that $U_{\alpha}\subsetneq W \subsetneq U_n$, the family $\{U_n\}_{n\in\mathbb{N}}$ is cofinal in $\mathring{\lceil U_{\alpha}\rceil}$. After all, for any $n \geq n_0$, the partition $\mathcal{K}'[\tilde{U}_n,\tilde{F}_n]$ was defined according to items \ref{item1} or \ref{item2}, where $[\tilde{U_n},\tilde{F_n}]\in \psi([U_n,F_n])$ is the basic open set containing $[U_{n+1},F_{n+1}]$. Then, we must have $U_n\subseteq W\subsetneq U_n$ for some big enough $n\geq n_0$, which is a contradiction;
        \item If the above item does not hold, we can assume that $n_0\in \mathbb{N}$ is big enough in order to the equality $U_n = U_{n_0}$ be verified for every $n \geq n_0$. In this case, we are under the hypothesis of Corollary \ref{mesmabase}, so that $U_{n_0}$ is the smallest clopen set from $\mathcal{C}$ containing $x$. However, this contradicts the assumption that $x\in W\subsetneq U_{n_0}$.
    \end{itemize}

    To summarize, there must exist $n\in \mathbb{N}$ such that $U_n\subseteq W$. On the other hand, given $\theta \in H$, we remarked in the first two paragraphs that $U_{\theta}\subsetneq U_m$ for every $m\in \mathbb{N}$ and $U_{\theta}\cap U_{\alpha} = \emptyset$ for every $\alpha \in \displaystyle \bigcup_{m \in \mathbb{N}}F_m$. In other words, we proved that $U_{\theta}\subset \displaystyle \bigcap_{m\in\mathbb{N}}[U_m,F_m]$. Hence, $U_{\theta}\subset A_x$ for every $\theta \in H$, because $x\notin U_{\theta}$. Since $\mathcal{K}[A_x]$ is singleton or a disjoint family of elements from $\mathcal{C}$, we observe that $\mathcal{F} = \{B \in \mathcal{K}[A_x]: U_{\theta}\subset B \neq \emptyset \text{ for some }\theta \in H\}$ is finite. Hence, $[U_n,F_n]\setminus \displaystyle \bigcup_{U\in\mathcal{F}} U\in \mathcal{V}_x$ and $\displaystyle x \in [U_n,F_n]\setminus \bigcup_{U\in\mathcal{F}} U\subset [W,H]$, finishing the proof. 
    
\end{proof}

Therefore, combining Proposition \ref{baselocal} and  Lemma \ref{descricao}, we see that $\psi(\mathcal{C})$ is a special clopen subbase for $X$. We will finish the proof that $X$ is the end space of a special tree, or, equivalently, the end space of some graph, via the recent characterization given by Pitz in \cite{pitz}. His main result claims that these spaces are precisely the ones that admit a $\sigma-$disjoint, nested, noetherian and \textit{hereditarily complete} clopen subbase.

In order to introduce this additional property, we recall that a family of sets $\mathcal{S}$ is \textbf{centered} if $\displaystyle \bigcap_{F\in\mathcal{F}}F \neq \emptyset$ for every finite $\mathcal{F}\subset \mathcal{S}$. Then, we say that $\mathcal{S}$ is \textbf{complete} if $\displaystyle \bigcap_{F\in\mathcal{F}}F \neq \emptyset$ for every centered family $\mathcal{F}\subset \mathcal{S}$. Finally, when $\mathcal{S}$ is a subfamily of the power set $\wp(X)$ of the topological space $X$, we call $\mathcal{S}$ \textbf{hereditarily complete} if $\{S\cap Y: S\in \mathcal{S}\}$ is complete for every closed subspace $Y\subset X$. Relying on Proposition \ref{baselocal}, its not hard to conclude that $\psi(\mathcal{C})$ has this property:

\begin{thm}\label{completude}
    The family $\psi(\mathcal{C})$ of the nodes of the tree $T_{\mathcal{C}}$ is hereditarily complete. 
\end{thm}
\begin{proof}
    Let $Y\subset X$ be a closed subspace and fix $\mathcal{S}_Y\subset \{Y\cap [U,F]: [U,F]\in\psi(\mathcal{C})\}$ a centered family. Then, $(Y\cap [W,H])\cap (Y\cap [V,G]) \neq \emptyset$ for any two open sets $[W,H],[V,G] \in \mathcal{S} = \{[U,F] \in \psi(\mathcal{C}): [U,F] \cap Y \in \mathcal{S}_Y\}$. In particular, since $\psi(\mathcal{C})$ is nested, $\mathcal{S}$ is a chain in $T_{\mathcal{C}}$. If its cofinality is finite, then $\mathcal{S}$ has a $\subseteq-$minimum element $[U,F]$, so that $\displaystyle \bigcap_{S\in \mathcal{S}}S = [U,F]\cap Y \neq \emptyset$.

    Assume now that $\mathcal{S}$ has infinite cofinality in $T$, so that $\displaystyle \bigcap_{S\in \mathcal{S}}S$ is written uniquely as a disjoint union $\{x\}\cup A_x$ for some $x\in X$. For a contradiction, suppose that $(\{x\} \cup A_x)\cap Y = \emptyset$. Therefore, since $X\setminus Y$ is open, by Proposition \ref{baselocal} there are $[U,F]\in \mathcal{S}$ and a finite family $\mathcal{F}\subset \mathcal{K}[A_x]$ such that $\displaystyle x \in [U,F]\setminus \bigcup_{F\in \mathcal{F}}F\subset X\setminus Y$. Hence, $[U,F]\subset X\setminus Y$, because $\displaystyle \bigcup_{F\in \mathcal{F}}F\subset A_x$. This, however, contradicts the definition of $S$. 
\end{proof}

In other words, Theorem \ref{completude} finishes a description of topological properties of a suitable clopen subbase $\psi(\mathcal{C})$ for $X$. More precisely, our study carried out in this section concludes that, if $X$ is a Hausdorff space in which $\mathrm{II}\uparrow \mathrm{End}_{\mathcal{B}}$ for some basis $\mathcal{B}$ generated by a given special clopen subbase $\mathcal{C}$, then there is $\psi(\mathcal{C})$ another special clopen subbase for $X$ that is hereditarily complete. According to Theorem $1.2$ of Pitz in \cite{pitz}, this proves that $X$ is the end space of a special tree. As formalized by Proposition \ref{estrategia}, Player $\mathrm{II}$ also has a winning strategy for the end game in these latter spaces, which, by Theorem $3$ of Kurkofka and Pitz in \cite{representacao}, are precisely the topological spaces that arise as end spaces of graphs. Therefore, Theorem \ref{t3} is established. 

\section{Choquet property}\label{choquetproperty} 
\paragraph{}
In this section, we present some relationships between the End game and the Banach-Mazur game; as an application, we show that every end space is productively Baire. A topological space is Baire if the intersection of any countable collection of open dense subset is dense. If the product $X \times Y$ is Baire, then $X$ and $Y$ must be Baire; however, the converse is not true in general. Indeed, Oxtoby \cite{Oxtoby}, under CH, and Cohen \cite{cohen}, in ZFC, constructed a Baire space with a non-Baire square.

Recall that the Banach-Mazur game on the space $X$ is played by two players, Player I and II, in $\omega$-many innings. At the beginning of the game, Player I chooses a nonempty open set $U_0$ and Player II responds by choosing a nonempty open set $V_0 \subset U_0$. In the $n$-th inning ( $n>0$ ), Player I chooses a nonempty $U_n \subset V_{n-1}$ and Player II responds by choosing a nonempty open set $V_n \subset U_n$, and so on. The Player II wins if $\bigcap_{n \in \omega} V_n \neq \emptyset$.

Banach, Mazur and Oxtoby proved that the space $X$ has the Baire property if and only if Player I does not have a winning strategy in the Banach-Mazur game on $X$, see for example \cite{oxtoby1957banach}. A space $X$ is called Choquet if Player II has a winning strategy in the Banach-Mazur game on $X$. Choquet spaces were introduced in 1975 by White who called them weakly $\alpha$-favorable spaces. Clearly, Choquet spaces are Baire. A Bernstein subset of reals witnesses that Baire spaces need not be Choquet. A productively Baire space is a Baire space whose product with every other Baire space is Baire. It is known that Choquet spaces are productively Baire.

The following propositions show some connection between the Banach-Mazur game and the End game.

\begin{prop}
    Let $X$ be a space that admits a special clopen (sub-)base $\mathcal{C}$. If Player I has a winning strategy in the Banach-Mazur game on $X$. Then Player I has a winning strategy in $End_{\mathcal{C}}(X)$.
\end{prop}
\begin{proof}
    It is known that playing the Banach-Mazur game using only elements of a fixed base is equivalent to the whole game \cite{Aurichi-Diaz}. Let $\sigma$ be a winnig strategy for Player I in the Banach-Mazur game on $\mathcal{C}$. Player II chooses $\mathcal{U}_0\subset \mathcal{C}$ that covers $\sigma(\langle\rangle)$. For every $U_0\in \mathcal{U}_0$, $\sigma(\langle U_0\rangle)\subseteq U_0$. Then Player II chooses $\mathcal{U}_{1,\langle U_0\rangle}\subset \mathcal{C}$ that covers $\sigma(\langle U_0\rangle)$ for every $U_0\in \mathcal{U}_0$. For every $U_1\in \mathcal{U}_1$, $\sigma(\langle U_0, U_1\rangle)\subseteq U_1$. Player II chooses $\mathcal{U}_{1,\langle U_0,U_1\rangle}\subset \mathcal{C}$ that covers $\sigma(\langle U_0,U_1\rangle)$ for every $(U_0, U_1)\in \mathcal{U}_0\times\mathcal{U}_1$, and so on for countably many innings. This generates a tree of possible evolution of the game in which Player I wins in every chain in the Banach-Mazur game restricted to the base $\mathcal{C}$. Therefore, for every $(U_n)_{n\in\mathbb{N}} \in \prod_{n\in \mathbb{N}}\mathcal{U}_n$, $\bigcap_{n\in\mathbb{N}} \sigma(\langle U_n : n\in \mathbb{N}\rangle)=\emptyset$, therefore Player I also wins the $End_{\mathcal{C}}(X)$. 
\end{proof}

\begin{prop}
    Let $X$ be a space that admits a special clopen (sub-)base $\mathcal{C}$. If Player II has a winning strategy in $End_{\mathcal{C}}(X)$. Then Player II has a winning strategy in the Banach-Mazur game on $X$.
\end{prop}
\begin{proof}
    Let $\sigma$ be a winning strategy for Player II in $End_{\mathcal{C}}(X)$. Let $A_0$ be the first choice of Player I in the Banach-Mazur game on $X$ played by picking only elements of $\mathcal{C}$. In the first inning of the Banach-Mazur game restricted on $\mathcal{C}$ Player I picks $U_0\in \mathcal{C}$. Then Player I chooses an element $V_0\in \sigma(\langle U_0\rangle)$. This $V_0$ is the first move of Player II in the Banach-Mazur game on $X$. Then Player chooses $U_1\subseteq V_0$ and also $V_1\in \sigma(\langle U_0,U_1\rangle)$. The subset $V_1$ is the second move of Player II in the Banach-Mazur game on $X$, and so on. Since $\sigma$ is a winning strategy for Player II in the $End_{\mathcal{C}}(X)$, $\bigcap_{n\in\mathbb{N}}V_n\not=\emptyset$. Therefore, Player II has a winning strategy in the Banach-Mazur on $X$. 
\end{proof}

\begin{question}
    Does there exist an example of a space $X$ with a special clopen (sub-)base $\mathcal{C}$,in which Player II has a winning strategy in the Banach-Mazur game and has not a winning strategy in the $End_{\mathcal{C}}(X)$?
\end{question}

\begin{cor}\label{baire}
    Every end space is Choquet, hence productively Baire.
\end{cor}
\begin{proof}
    Let $\mathcal{C}$ be the special base on $X$. By  Theorem \ref{t3}, Player II has a winning strategy in the $End_{\mathcal{C}}(X)$,  then Player II has also a winning strategy in the Banach-Mazur game on $X$. Therefore $X$ is productively Baire.
\end{proof}

\begin{thm}\label{t4}
   Ray and Branch spaces are Choquet, hence productively Baire.
\end{thm}
\begin{proof}
    In \cite{pitz} it is shown that both ray and branch spaces have a complete special base. It is known that the Banach-Mazur game, played only choosing elements from a specific base of the space, is equivalent to the Banach-Mazur game on the space (see, for instance, \cite{Aurichi-Diaz}). Therefore, it is straightforward to see that Player II has a winning strategy on the Banach-Mazur game on every ray and branch space. 
\end{proof}

\section{Subspaces}
\label{subspaces}
\paragraph{}
In the following, we try to study how the subspace relation ask $A \hookrightarrow X$, where $X$ is either a branch or end-space, translates into the combinatorial realm. We must firstly find which sub-spaces of end-spaces are end-spaces themselves. The main result of this section claims that every $G_{\delta}$ subspace of an end space is also an end space; we obtain this as an application of Theorem \ref{t3}.

If the end-space of a graph $G$ is metrizable, it will be by Theorem 3.1 of Pitz in \cite{pitz} in fact completely ultra-metrizable. For trees of this height, a natural way of defining a graph structure on $T$ is available. We shall denote such graph still with $T$ and, in fact, it is easy to see that in this case $\Omega(T)$ will coincide with all of the above spaces. We state the following theorem one can find in \cite{kechris2012classical}:

\begin{thm}[Lavrentiev]
    The completely metrizable sub-spaces of a completely metrizable space are exactly its $G_\delta$ subspaces
    , i.e. countable intersections of open sets.
\end{thm}

That, together with the previous results, gives us

\begin{prop}
\label{metcharacterization}
    The sub-spaces of metrizable end-spaces $\Omega(G)$ that are end-spaces are precisely its $G_\delta$ sub-spaces.
\end{prop}

\begin{proof}
    A $G_\delta$ set $A \subset \Omega(G)$ is by Lavrentiev's theorem completely metrizable and zero-dimensional, since this is a hereditary property. By Corollary 5 of \cite{ultrametrization} any zero-dimensional complete metrizable space $(X,d)$ admits an equivalent complete ultra-metric $d'$, and by  Proposition 2.1 of \cite{brian} will be $A \approx \Omega(T_G) = \mathcal{R}(T_G) = \mathcal{B}(T_G)$ for some pruned tree $T_G$ of height $\omega$.

    Now let $A \subset \Omega(G)$ be an end-space $A \approx \Omega(\tilde{G})$ for some $\tilde{G}$, then it is a metrizable end-space. By Theorem 3.1 of Pitz in \cite{pitz} $A$ must be in particular completely metrizable, and again by Lavrentiev's theorem $A$ must be a $G_\delta$ set.
\end{proof}

\begin{cor}
    Bernstein subsets $B \subset \mathcal{C}$ are Baire but not end-spaces.
\end{cor}

The above can also be concluded using Corollary \ref{baire}, since Bernstein sets are not Choquet. The characterization suggests we should look into $G_\delta$ sets. The following lemma is obtained as an application of Theorem \ref{t3}, combined with Theorem 1 of Kurkofka, Melcher and Pitz in \cite{approximatinginfinitegraphsbynormaltrees}:

\begin{thm}[\cite{approximatinginfinitegraphsbynormaltrees}]
\label{normalapprox}
    Every open covering can be refined by a pairwise disjoint clopen partition of basics $\{\Omega(C)\}_{C \in \mathrm{S}_{G,T}}$ where $T$ is normal and rayless. In particular, end-spaces are ultra-paracompact.
\end{thm}

\begin{lemma}\label{endlemma}
    Let $X$ be an end space and $\mathcal{B}=\{\Omega(C)\}_{C \in \mathrm{S}_{G,T}}$ be the usual base of $X$. The following statements about the game $End$ are equivalent:
    \begin{enumerate}
        \item Player II has a stationary winning strategy in $End_{\mathcal{B}}(X)$;
        \item Player II has a stationary winning strategy in $End(X)$, where $End(X)$ is the End game without the restriction of playing open sets from a fixed base.
    \end{enumerate}
\end{lemma}

\begin{proof}
    $(1\Rightarrow 2)$: Let $\psi$ be the stationary winning strategy of Player II in the game $End_{\mathcal{B}}(X)$. We will show that Player II has a stationary winning strategy in the game $End(X)$. Let us play a round of the game $End(X)$:
    \begin{itemize}
        \item In round 0, let $U_0\subset X$ be the non-empty open set played by Player I in $End(X)$. By Theorem \ref{normalapprox}, there is a cover $\mathcal{C}_0$ consisting of pairwise disjoint open sets from the base $\mathcal{B}$ for the open set $U_0$. For each $V_0\in \mathcal{C}_0$, consider $\mathcal{C}_0^{V_0}=\psi (\langle V_0\rangle)$. Then, we define $\mathcal{U}_0=\bigcup_{V_0\in\mathcal{C}_0} \mathcal{C}_0^{V_0}$ as Player II's response in round 0 of the game $End(X)$.
        \item In round $n\in\omega$, let $U_n\subset W_n\in \mathcal{U}_{n-1}$ be the non-empty open set played by Player I in $End(X)$. So, $U_n\subset W_{n-1}\subset V_{n-1}\subset U_{n-1}$, where $W_{n-1}, V_{n-1}\in \mathcal{B}$. By  Theorem \ref{normalapprox}, there is a cover $\mathcal{C}_n$ consisting of pairwise disjoint open sets from the base $\mathcal{B}$ for the open set $U_n$. For each $V_n\in \mathcal{C}_n$, consider $\mathcal{C}_n^{V_n}=\psi (\langle V_n\rangle)$. Then, we define $\mathcal{U}_n=\bigcup_{V_0\in\mathcal{C}_n} \mathcal{C}_n^{V_n}$ as Player II's response in round $n$ of the game $End(X)$.
    \end{itemize}
    At the end of the game,
    $$\langle U_0, \mathcal{U}_0, \dots, U_n, \mathcal{U}_n, \dots \rangle$$
    played in the game $End(X)$, we obtain $\bigcup_{n\in\omega}U_n=\bigcup_{n\in\omega} V_n$, where $V_n\subset U_n\subset W_{n-1}\subset V_{n-1}\subset U_{n-1}$. Since 
    $$\langle V_0, \mathcal{C}_0^{V_0}, \dots, V_n, \mathcal{C}_n^{V_n}\dots \rangle$$
    is a game played in $End_{\mathcal{B}}(X)$ where Player II used the stationary winning strategy $\psi$, we have $\bigcup_{n\in\omega} V_n= \lbrace x \rbrace \cup A$, where $A$ is an open set of $X$ and $x\notin A$. Furthermore, $x\in X$ and $A$ are the only ones with this configuration. Therefore, Player II has a stationary winning strategy in $End(X)$.

    $(2\Rightarrow 1)$: Let $\varphi$ be the stationary winning strategy of Player II in the game $End(X)$. We will show that Player II has a stationary winning strategy in the game $End_{\mathcal{B}}(X)$. Let us play a round of the game $End_{\mathcal{B}}(X)$:
    \begin{itemize}
        \item In round 0, let $U_0\in \mathcal{B}$ be the non-empty open set played by Player I in $End_{\mathcal{B}}(X)$. Consider $\varphi (\langle U_0 \rangle)=\mathcal{C}_0$. By Theorem \ref{normalapprox}, for each $V_0\in \mathcal{C}_0$ there exists a cover $\mathcal{C}_0^{V_0}$ consisting of pairwise disjoint open sets from the base $\mathcal{B}$ for the open set $V_0$. Then, we define $\mathcal{U}_0=\bigcup_{V_0\in\mathcal{C}_0} \mathcal{C}_0^{V_0}$ as Player II's response in round 0 of the game $End_{\mathcal{B}}(X)$.
        \item In round $n\in\omega$, let $U_n\subset W_n\in \mathcal{U}_{n-1}$ be the non-empty open set played by Player I in $End_{\mathcal{B}}(X)$. So, $U_n\subset W_{n-1}\subset V_{n-1}\subset U_{n-1}$, where $W_{n-1}, U_n, U_{n-1}\in \mathcal{B}$. Consider $\varphi (\langle U_n \rangle)=\mathcal{C}_n$. By Theorem \ref{normalapprox}, for each $V_n\in \mathcal{C}_n$, there exists a cover $\mathcal{C}_n^{V_n}$ consisting of pairwise disjoint open sets from the base $\mathcal{B}$ for the open set $V_n$. Then, we define $\mathcal{U}_n=\bigcup_{V_n\in\mathcal{C}_n} \mathcal{C}_n^{V_n}$ as Player II's response in round $n$ of the game $End_{\mathcal{B}}(X)$.
    \end{itemize}
    In the end, we play a game in $End(X)$,
    $$\langle U_0, \mathcal{C}_0, \dots, U_n, \mathcal{C}_n, \dots \rangle$$
    and a game in $End_{\mathcal{B}}(X)$,
    $$\langle U_0, \mathcal{U}_0, \dots, U_n, \mathcal{U}_n, \dots \rangle,$$
    in which Player I's responses are the same. Since in $End(X)$, Player II used their stationary winning strategy $\varphi$, we have $\bigcup_{n\in\omega}U_n=\lbrace x \rbrace \cup A$, where $A$ is an open set of $X$ and $x\notin A$. Furthermore, $x\in X$ and $A$ are the only ones with this configuration. Therefore, Player II has a stationary winning strategy in $End_{\mathcal{B}}(X)$.
\end{proof}

Proposition \ref{metcharacterization} presents a complete topological characterization of the subspaces of a metrizable end space that are also end spaces. However, the general case remains an open problem.

\begin{question}
    Which subspaces of an end space are also end spaces?
\end{question} 

To conclude this section, we present the main theorem regarding subspaces, which contributes to the aforementioned open question. To this end, we will use the following lemma.
 
\begin{lemma}[\cite{approximatinginfinitegraphsbynormaltrees}]\label{opensareends}
Open subsets of end spaces are still end spaces.
\end{lemma}

\begin{thm}
\label{generalcasegdelta}
   Let $A \subset \Omega(G)$ be a $G_{\delta}$ subset, then $A$ is an end-space.
\end{thm}

\begin{proof}
Consider $\mathbb{N}=\bigcup_{i\in\omega}N_i$ as a partition of the naturals into disjoint infinite sets. Since each $A_n$ is open, by  Lemma \ref{opensareends}, each $A_n$ is an end space. By Lemma \ref{endlemma}, for each $n\in\omega$, there exists a stationary winning strategy $\phi_n$ for Player II in the game $End(A_n)$. Let $\mathcal{B}'$ be the special basis of $A$ induced by the special basis of $X$. We will show that Player II has a stationary winning strategy in the game $End_{\mathcal{B}'}(A)$. To prove this, we will play a round of the game $End_{\mathcal{B}'}(A)$.

\begin{itemize}
    \item In round 0, let $U_0\in \mathcal{B}'$ be the non-empty open set played by Player I in the game $End_{\mathcal{B}'}(A)$. Since $\mathbb{N}=\bigcup_{i\in\omega}N_i$ is a decomposition, there exists $i_0\in\omega$ such that $0\in N_{i_0}$. Consider $U_0^{i_0}\subset A_{i_0}$ an open set such that $U_0^{i_0}\cap A=U_0$. Let $\phi_{i_0}(\langle U_0^{i_0}\rangle)=\mathcal{C}_0^{i_0}$. By Theorem \ref{normalapprox}, for each $V_0\in\mathcal{C}_0^{i_0}$, consider $\mathcal{C}_0^{i_0, V_0}$ a covering formed by pairwise disjoint open sets from the basis $\mathcal{B}$ for the open set $V_0$. Define $\mathcal{U}_0=\bigcup_{V_0\in\mathcal{C}_0^{i_0}} \lbrace V\cap A : V\in \mathcal{C}_0^{i_0, V_0}\rbrace$ as Player II's move in round 0 of the game $End_{\mathcal{B}'}(A)$.
    
    \item In round $n$, let $U_n\in \mathcal{B}'$ be the non-empty open set played by Player I in the game $End_{\mathcal{B}'}(A)$. Then $U_n\subset W_n \subset U_{n-1}$, where $ W_n\in \mathcal{U}_{n-1}$. Note that there exist unique $V_n\in \mathcal{C}_{n-1}^{i_{n-1}}$ and $W_n'\in \mathcal{C}_{n-1}^{i_{n-1},V_{n-1}}$ such that $W_n=W_n'\cap A$ and $W_n'\subset V_n$. Since $\mathbb{N}=\bigcup_{i\in\omega}N_i$ is a decomposition, there exists $i_n\in\omega$ such that $n\in N_{i_n}$. Consider $U_n^{i_n}\subset A_{i_n}$ an open set such that $U_n^{i_n}\cap A=U_n$ and $U_n^{i_n}\subset W_{n}'$. Let $\phi_{i_n}(\langle U_n^{i_n}\rangle)=\mathcal{C}_n^{i_n}$. By Theorem \ref{normalapprox}, for each $V_n\in\mathcal{C}_n^{i_n}$, consider $\mathcal{C}_n^{i_n,V_n}$ a covering formed by pairwise disjoint open sets from the basis $\mathcal{B}$ for the open set $V_n$. Define $\mathcal{U}_n=\bigcup_{V_n\in\mathcal{C}_n^{i_n}} \lbrace V\cap A : V\in \mathcal{C}_n^{i_n,V_n}\rbrace$ as Player II's move in round $n$ of the game $End_{\mathcal{B}'}(A)$.
\end{itemize}

At the end of the game, we obtain $\bigcup_{n\in\omega}U_n=(\bigcup_{k_n\in N_k}U_{k_n}^{i_k})\cap A\subset A_k$, for all $k\in\omega$. Since
$$\langle U_{k_0}^{i_k}, \mathcal{U}_{k_0}^{i_k}, U_{k_1}^{i_k}, \mathcal{U}_{k_1}^{i_k}, \dots, U_{k_n}^{i_k}, \mathcal{U}_{k_n}^{i_k}, \dots \rangle $$
is a valid play in $End(A_k)$ where Player II used their stationary winning strategy $\phi_k$, there exist unique $x_k\in A_k$ and an open set $S_k\subset A_k$ such that $x_k\notin S_k$ and $\bigcup_{k_n\in N_k}U_{k_n}^{i_k}=\lbrace x_k\rbrace \cup S_k$. Thus,
$$\bigcup_{n\in\omega}U_n=(\lbrace x_k\rbrace \cup S_k)\cap A\subset A_k$$
for all $k\in\omega$. Suppose there exists $y\in A_k\setminus A$, for some $k\in\omega$. Since $A$ is $G_{\delta}$, there exists $n\in\omega$ such that $y\notin A_n$. Given that $\mathbb{N}=\bigcup_{i\in\omega}N_i$ is a decomposition, there exists $s\in\mathbb{N}$ such that $s\in N_n$, thus $U_s^{i_s}\subset A_n$. By construction, for all $m\in\mathbb{N}$ such that $m\geq s$, $U_m^{i_m}\subset U_s^{i_s}$. Therefore, $y\notin \bigcup_{k_n\in N_k}U_{k_n}^{i_k}$, which is a contradiction! Hence, $y\in A$ for all $k\in\omega$. Therefore, $\lbrace x_k \rbrace \cup S_k\subset A$ for all $k\in\omega$. By the uniqueness of the expression $\bigcup_{k_n\in N_k}U_{k_n}^{i_k}=\lbrace x_k\rbrace \cup S_k$, we obtain $x=x_k=x_l$ and $S=S_k=S_l$ for all $k,l\in\omega$. Thus, $\bigcup_{n\in\omega}U_n=\lbrace x\rbrace \cup S$. Therefore, Player II has a stationary winning strategy in the game $End_{\mathcal{B}'}(A)$.
\end{proof}

\section{Topological products}
\label{product}
\paragraph{}
In this section, we study the topological space resulting from the product of end spaces. We first address the case of the product of metrizable end spaces in Section \ref{metrizablecaseproduct}, where the structure is well-behaved. Subsequently, in Section \ref{obstructions}, we treat the product of general end spaces. In this latter case, the product of end spaces is not always an end space; a counterexample is provided as an application of Theorem \ref{t3}

\subsection{The metrizable case}
\label{metrizablecaseproduct}

\paragraph{}
If the end-space of a graph $G$ is metrizable, it will be by Theorem 3.1 of Pitz in \cite{pitz} in fact completely ultra-metrizable and therefore by Proposition 2.1 of Brian in \cite{brian} will be $\Omega(G) = \Omega(T_G) = \mathcal{R}(T_G) = \mathcal{B}(T_G)$ for some pruned tree $T_G$ of height $\omega$. For trees of this height, a natural way of defining a graph structure on $T$ is available. We shall denote such graph still with $T$ and, in fact, it is easy to see that in this case $\Omega(T)$ will coincide with all of the above spaces. Any zero-dimensional complete metric space $(X,d)$ admits  an equivalent complete ultra-metric $d'$ (Corollary 5, \cite{ultrametrization}), therefore metrizable branch-spaces also falls in this class.

Consider a countable product $\prod_{n \in \mathbb{N}}X_n$ of completely ultra-metrizable spaces. Complete metrizability and zero-dimensionality is preserved by countable product, therefore $\prod_{n \in \omega} X_n$ is homeomorphic to $\Omega(T)$ for some pruned tree of height $\omega$, therefore

\begin{prop}
    \label{countablemetrizable}
    The class of metrizable branch- and end-spaces is closed under countable products.
\end{prop}

The above proof is very indirect, and do not construct a new graph $G$ such that $\Omega(G)=\prod_{i \in \mathbb{N}}\Omega(G_n)$. In the metrizable case we have $T_G \leq G$ a normal faithful tree of the graph. In the following, we attempt to construct a natural $\prod T_{G_n}$ that realizes the desired end-space $\Omega(\prod T_{G_n}) \approx \prod \Omega(G_n)$ for the finite product case.

Notice a pruned tree $T$ of ordinal height $\omega$ can be seen as a sequence $T^\bullet:\omega \to \catname{Top}$ of topological spaces of it is levels $T^n \doteq \{t \in T \st t \text{ has height }n\}$, which we consider with discrete topology. For every pair $n \leq m $, every $x \in T^m$ has an unique ancestral $a^{nm}(x) \in T^n$ in that lower level. Such systems of topological-spaces, coupled with maps $(T^\bullet,\{a_{nm}:T^m \to T^n\}_{m \geq n})$ such that $a^{m \ell} \circ a^{nm} = a^{n \ell}$ and $a^{nn} = \mathrm{id}_{T^n}$ are called \textbf{inverse systems}, or rather \textbf{diagrams} in a more general categorical setting, which we shall not consider. 

Consider $T^\bullet$ an inverse system, we define an associated tree $(T, \leq)$ in the following. the set $T$ given by the disjoint union of its levels. For every $t \in T^m_i$ and $n \leq m$, it is well defined the unique $a_{i}^{nm}(x) \in T^n$ the only element $t' \in T^n_i$ of this level such that $t' \leq t$. Notice for any $(t_i) \in T^n_i$, we have that \[\lceil (t_i) \rceil = \{(a^{n-1,n}_i(x_i))>\dots>(a^{1,n}_i(x_i))>(r_{i})\}\] is well-ordered, where $r_{i}$ is the root of $T_i$, providing $T$ with an order that makes it into an order theoretic tree. 

Let $(T^\bullet,\{a^{nm}\}_{m \geq n})$ be an inverse system. Consider $P \doteq \prod_{n \in \omega} T^n$ with product topology, with canonical projections $\{\pi_{n}\}_{n \in \omega}$. We define the \textbf{the projective(inverse) limit} of this system of topological spaces as
\[
\varprojlim T^\bullet = \{(x_n) \in P \st \forall n \leq m (a^{nm}(x_m) = x_n)\} \subset \prod_{n \in \omega} T^n
\]
It is not difficult to show that we have the homeomorphism $\varprojlim T^\bullet \approx \mathcal{B}(T)$ where $T^\bullet$ is the system associated with $T$. It is important to notice if $T^\bullet$ is a general inverse system, its levels may not be discrete, therefore $\varprojlim T^\bullet$ will not be homeomorphic to the branch-space $\mathcal{B}(T)$ of its associated tree. 

In the following we describe a product of trees as systems.

\begin{defn}[Level-wise product of trees]
\label{system-product}
For $i \in I \neq \emptyset$ any index set, let $(T_i^\bullet:\omega \to \catname{Top},\{a_i^{nm}:T_i^m \to T_i^n\}_{m \geq n})$ be a pruned tree of height $\omega$, given as a system. Define $T^\bullet \doteq \prod_{i \in I}T_i^\bullet$ as the following \textbf{system}:
\begin{itemize}
    \item At $T^n \doteq \prod_{i \in T}T_i^n$ with product topology. Each level assumes natural projections $\pi^n_i:T^n \to T_i^n$.
    \item The maps $a^{nm}:T^m \to T^n $ are the tuple of maps $\langle a^{mn}_i \rangle_{i \in I}$.
\end{itemize}
\end{defn} 

What we get from this construction is that the limit of the product is the product of the limits:

\begin{prop}[Direct application of 2.5.10 of \cite{engelking}]
\label{branch-product-preservation}
    Consider a family of inverse systems $(T_i^\bullet:\omega \to \catname{Top},\{a_i^{nm}:T_i^m \to T_i^n\}_{m \geq n})$, then \[\varprojlim \left( \prod_{i \in I} T^\bullet_i\right) = \prod_{i \in I}(\varprojlim T_i^\bullet).\]
\end{prop} 

\begin{cor}
    If $I$ is finite, the tree $T$ associated with the system $\prod_{i \in I} T_i^\bullet$ has $\Omega(T) = \prod_{i \in I}\Omega(T_i)$ when $T_i$ is pruned of height $\omega$.  
\end{cor}

In particular, there is a binary operation $\otimes$ in pruned trees of height $\omega$ such that $\mathcal{B}(T_1 \otimes T_2) \overset{\text{(i)}}{\approx} \mathcal{B}(T_1)\times \mathcal{B}(T_2)$ and $\Omega(T_1 \otimes T_2) \overset{\text{(ii)}}{\approx} \Omega(T_1)\times \Omega(T_2)$. For general trees, (i) is proven false in Example \ref{branch-counter-example} and (ii) in Theorem \ref{end-counter-example}.

\begin{ex}[Failure for infinitely many trees]
    But consider the complete binary tree $B $ whose branch set is the Cantor space $\mathcal{C}$. We will have in fact that $\varprojlim (B^\bullet)^\mathbb{N} = \mathcal{C}^\mathbb{N} = \mathcal{C}$ but the tree $T$ associated to the system $(B^\bullet)^\mathbb{N} $ is $\mathbb{N}^\mathbb{N}$, and its branch set is the Baire space $\mathcal{B}(T) = \mathbb{N}^\mathbb{N}$.
\end{ex}

\subsection{Counter-examples}
\label{obstructions}
\paragraph{}
In the following, we tackle some counter-examples of product closure on branch and end-spaces in the general case. We begin with a counter example for Proposition \ref{countablemetrizable} in the general case:

\begin{ex}[Uncountable products will not work]
\label{naturalpower}
    The natural numbers $\mathbb{N}$ is a completely ultrametrizable (branch-/end-)space where $\mathbb{N}^{\aleph_1}$ is not normal by Theorem 3 of Stone in \cite{Stone}.
\end{ex}

So branch and end-spaces are already not closed under uncountable products.

\begin{ex}[Michael line $\mathbb{M}$ as a branch space]
\label{michael}
    We say a ray of $2^{< \omega}$ is \textbf{rational} if it is eventually constant, \textbf{irrational} otherwise. Consider $T_\mathbb{M}$ a tree of height $\omega+\omega$ consisting of the binary tree $2^{<\omega}$, added with with two tops above each \textbf{irrational} branch and exactly one successor for each element of height $\geq \omega$. The branch  space $\mathcal{B}(T_\mathbb{M}) = \mathbb{M}$ is the Michael line.
\end{ex}

It is known (\cite{NYIKOS19991}) that branch spaces are ultraparacompact and Hausdorff, therefore normal. The above example then gives us the following counter-examples of product closure of branch spaces:

\begin{ex}
    \label{branch-counter-example}
    Consider the following products:
    \begin{itemize}
        \item The power of the Michael line with itself infinitely many times $\mathbb{M}^\omega$ is not normal, therefore it cannot be a branch-space. 
        \item Irrationals $\mathbb{R} \backslash \mathbb{Q} \approx \mathcal{B}(\omega^{<\omega})$ is a completely ultrametrizable branch set. The Michael line $\mathbb{M} \times (\mathbb{R} \backslash \mathbb{Q})$ is not a normal space, therefore it cannot be the branch space of a tree.
    \end{itemize}
\end{ex}

In particular, there \textbf{cannot} be a binary $\otimes$ over special trees such that $\mathcal{B}(T_1 \otimes T_2) = \mathcal{B}(T_1)\times \mathcal{B}(T_2)$.

\begin{question}
 What is the class of ray spaces (ends, branches) that is closed under products? That is, for which class is the product of any two ray spaces (ends, branches) still a ray space (ends, branches)? 

\end{question}

We say a vertex $v$ \textbf{dominates a ray} $r$ if there are infinitely many vertex-disjoint paths from $v$ to $r$, and it \textbf{dominates an end} when it dominates one of its representatives. 

\begin{ex}\label{ex}
    Consider $G$ a graph consisting of a collection of $\aleph_1$ vertex-disjoint rays. We choose one of the rays $\xi$ to be a distinguished one, and add edges from the first vertex of every $\xi' \neq \xi$ so that they dominate $\xi$. Lets identify $\Omega(G)\approx \omega_1$, and the distinguished $\xi $ with $0 \in \omega_1$. The topology here is the discrete one about every point $\alpha \in \omega_1 \backslash \{0\}$ and the basics from the point $0$ are $0 \cup (\omega_1 \backslash F)$ where $F \subset \omega_1$ is finite. This is a compact Hausdorff space.
\end{ex}

\begin{figure}[!ht]
\centering
\begin{tikzpicture}
    \def\radius{0.15}
    \def\dy{0.8}
    \def\dx{1}
    \foreach \x in {0,1,2,3,4,5} {
        \draw[->] (\x*\dx,\dy) -- (\x*\dx,2*\dy) ;
    }
    \node[right] at (6*\dx,1.5*\dy) {$\overset{\aleph_1 \text{-many copies}}{\cdots}$};
    \draw[->] (2.5*\dx,-2*\dy) -- (2.5*\dx,0) node[above] {$\xi$};
    \foreach \x in {0,1,2} {
        \foreach \y in {1,...,5} {
            \draw (\x*\dx,\dy) to[out=-45,in=180] ({2.5*\dx},{(-2/\y)*\dy});
        }
    }
    \foreach \x in {3,4,5} {
        \foreach \y in {1,...,5} {
            \draw (\x*\dx,\dy) to[out=225,in=0] ({2.5*\dx},{(-2/\y)*\dy});
        }
    }
\end{tikzpicture}
\end{figure}

In the following, we assume any special complete clopen base $\mathcal{B}$ is closed under finite intersections and complements. 

\begin{lemma}
\label{refinementforcompactendspace}
    Let $X$ be a compact end-space and $\mathcal{B}$ be a special complete clopen base of its topology, then for any open covering $\mathcal{O}$ there is a pairwise disjoint $\mathcal{B}$-covering that refines $\mathcal{O}$.
\end{lemma}

\begin{proof}
    As $X$ is ultra-paracompact, we have a (cl)open partition $\mathcal{P} \prec \mathcal{O}$. Fix an element $P \in \mathcal{P}$. Take $x \in P$, as $\mathcal{B}$ is a basis and $P$ is open, there is $x \in B_x \in \mathcal{B}$ such that $B_x \subset P$. Now $\{B_x\}_{x \in P}$ is an open cover of the compact $P$. Passing to a finite covering, we get a finite open covering $\{B_1,\dots,B_n\}$ of $P$ with basics from $\mathcal{B}$. Now, $\tilde{B}_i \doteq B_i \backslash \left( \bigcup_{j \neq i}B_j \right)$ is an intersection of basic opens.
\end{proof}

\begin{thm}\label{thm3}
    Let $(X,\tau)$ be a Hausdorff  compact ultra-paracompact topological space. If there is a clopen base $\mathcal{B}$ for $\tau$ in which $I\uparrow End_{\mathcal{B}}(X)$, then $X$ is not an end space.  
\end{thm}

\begin{proof}
    Let $\mathcal{B}',\mathcal{B}$ be bases for $\tau$ where $\mathcal{B}$ is such that $I\uparrow End_{\mathcal{B}}(X)$. Consider $\phi$ a a winning strategy of player $I$ in $End_{\mathcal{B}}$. We must prove Player II has no winning strategy in $End_{\mathcal{B}'}(X)$. Let $\psi$ a strategy from Player II in $End_{\mathcal{B}'}(X)$.

    \begin{itemize}
        \item Consider $\mathcal{U}_0$ a pairwise disj. $\mathcal{B}'$-covering for the open $\phi (\langle \rangle)$. For each $U\in \mathcal{U}_0$ Player II has an answer $\psi(\langle U\rangle)$. Notice $\mathcal{A}_0=\bigcup \lbrace \psi(\langle U\rangle): U\in \mathcal{U}_0\rbrace$ is an open covering with pairwise disjoint elements of $\mathcal{B}'$ for  $\phi(\langle \rangle)$. Using Lemma \ref{refinementforcompactendspace}, consider $\mathcal{V}_0$ a refinement of the covering $ \mathcal{A}_0$ by pairwise disjoint elements of $\mathcal{B}$. 

        \item  Notice $\mathcal{V}_0$ is an available move for Player II in $End_{\mathcal{B}}$, therefore it must be in the domain of the strategy $\phi$. Hence, $\phi(\langle \mathcal{V}_0\rangle)\subset U_0$ for a unique $U_0\in \mathcal{U}_0$. Consider $\mathcal{U}_1$ a pairwise disjoint open $\mathcal{B}'$-cover for $\phi(\langle \mathcal{V}_0\rangle)$. For each $U\in \mathcal{U}_1$ notice Player II has the answer $\psi(\langle U_0,U\rangle)$. We also have that $\mathcal{A}_1=\bigcup\lbrace \psi(\langle U_0,U\rangle ): U\in \mathcal{U}_1\rbrace$ is a pairwise disjoint open $\mathcal{B}'$-covering for $\phi(\langle \mathcal{V}_0\rangle)$. Using Lemma \ref{refinementforcompactendspace} consider $\mathcal{V}_1$ a refinement of the covering $\mathcal{A}_1$ by pairwise disjoint opens from $\mathcal{B}$. 
        
        \item Consider the history of moves of both games up until the the $n$-th turn of the game: $$\langle \phi(\langle \rangle), \mathcal{V}_0, \phi(\langle \mathcal{V}_0\rangle), \mathcal{V}_1, \cdots , \phi(\langle \mathcal{V}_0, \mathcal{V}_1, \cdots , \mathcal{V}_{n-1}\rangle), \mathcal{V}_{n}\rangle $$ 
        $$ \langle U_0, \mathcal{A}_0, U_1, \mathcal{A}_1, \cdots , U_{n-1}, \mathcal{A}_{n-1} \rangle$$ 
        Notice that $\mathcal{V}_{n}$ is an available move for Player II in the game $End_{\mathcal{B}}$, it is hence in the domain of $\phi$. In this way $\phi(\langle \mathcal{V}_0, \mathcal{V}_1, \cdots , \mathcal{V}_{n}\rangle)\subset U_{n}$ for a unique $U_n\in \mathcal{U}_n$. Consider $\mathcal{U}_{n}$ a pairwise disjoint $\mathcal{B}'$-covering for the open $\phi(\langle \mathcal{V}_0, \mathcal{V}_1, \cdots , \mathcal{V}_{n}\rangle)$. For each $U\in \mathcal{U}_{n}$ notice Player II has an answer $\psi(\langle U_0,U_1, \cdots , U_n, U\rangle)$. Now we must have that \[\mathcal{A}_n=\bigcup\lbrace \psi(\langle U_0,U_1, \cdots , U_n, U\rangle): U\in \mathcal{U}_{n}\rbrace\] is a pairwise disjoint open $\mathcal{B}'$-covering for $\phi(\langle \mathcal{V}_0, \mathcal{V}_1, \cdots , \mathcal{V}_{n}\rangle)$. Now using Lemma~\ref{refinementforcompactendspace} consider $\mathcal{V}_{n+1}$ a refinement of $\mathcal{A}_n$ by pairwise elements of $\mathcal{B}$.

    \end{itemize}

    Hence, we have $\bigcap_{n\in\omega} U_n = \bigcap_{n\in\omega} \phi(\langle \mathcal{V}_0, \mathcal{V}_1, \cdots , \mathcal{V}_{n}\rangle)$. As $\phi$ is winning in $End_{\mathcal{B}}(X)$ by hypothesis, there are unique $x\in X$ and $A\in \tau$ with $x\notin A$ such that $\bigcap_{n\in\omega} U_n = \lbrace x\rbrace \cup A$. Therefore there is a run in $End_{\mathcal{B}'}(X)$ such that Player II uses $\psi$ and loses, witnessing $\psi$ is not winning. We conclude Player II does not win $End_{\mathcal{B}'}$ for every base $B'$ of $\tau$.
\end{proof}

\begin{thm}
\label{end-counter-example}
    End-spaces are not closed under finite products.
\end{thm}

\begin{proof}
    Consider $T$ the complete binary tree of height $\omega$ with the Cantor space $\mathcal{C}$ as its end-space. For each $t\in T$, we have that $\stackrel{\circ}{\lceil t\rceil}= \lbrace s\in T : s<_{T} t\rbrace$ is finite. Moreover, for each $t\in T$, $\lfloor t\rfloor$ is a connected component of $T\setminus \stackrel{\circ}{\lceil t\rceil}$ defining a basic open $\Omega (\stackrel{\circ}{\lceil t\rceil}, \lfloor t\rfloor)$ in $\Omega (T)$. In particular, $\lbrace \Omega (\stackrel{\circ}{\lceil t\rceil}, \lfloor t\rfloor) : t\in T\rbrace$ is a clopen base for $\Omega (T)$. Consider now $(\omega_1, \tau )$ the end-space from Example \ref{ex}. Let us write explicitly a base for $(\omega_1,\tau)$ with $\lbrace A_F= \omega_1\setminus F: F\subset \omega_1\setminus \lbrace 0\rbrace \text{ with } F \text{ finite } \rbrace \cup \lbrace A_{\alpha}= \lbrace \alpha : \alpha \in \omega_1 \setminus \lbrace 0 \rbrace \rbrace $. Consider their topological product $X=\Omega (T)\times (\omega_1,\tau)$, with base \[\mathcal{B}= \lbrace \Omega (\stackrel{\circ}{\lceil t\rceil}, \lfloor t\rfloor) \times A_F, \Omega (\stackrel{\circ}{\lceil t\rceil}, \lfloor t\rfloor)\times A_{\alpha} \st t\in T, F\subset [\omega_1\setminus \lbrace 0 \rbrace]^{<\aleph_0}, \alpha\in\omega_1\setminus \lbrace 0 \rbrace \rbrace .\] Let us prove that $I\uparrow End_{\mathcal{B}}(X)$. 

    \begin{itemize}
        \item We start with a branch $R=\lbrace t_i : i\in\mathbb{N}\rbrace$ with $t_i<t_{i+1}$ from the binary tree $T$ corresponding to an end $\varepsilon$ of $T$. Let us say Player I opens the game with $U_0=\Omega (\stackrel{\circ}{\lceil t_0\rceil}, \lfloor t_0\rfloor)\times A_{\emptyset}\in \mathcal{B}$. Player II will then answer with a pairwise disjoint $\mathcal{B}$-covering $\mathcal{U}_0$ for the open $U_0$. 

        \item Let us say Player I's now chooses the open $V_1\in \mathcal{U}_0$ containing $(\varepsilon, 0)$. As $V_1=W_1\times Z_1$ has the point $(\varepsilon, 0)$ we know then $W_1=A_{F_1}$ for some $F_{k_1}\subset [\omega\setminus \lbrace 0 \rbrace ]^{<\aleph_0}$. Now Player I chooses an open $U_1=\Omega (\stackrel{\circ}{\lceil t_{n_1}\rceil}, \lfloor t_{n_1}\rfloor)\times A_{F_{s_1}}\subset V_1$ such that $t_0< t_{n_1}$ and $F_{k_1}\subset F_{s_1}$. Player II moves with a pairwise disjoint $\mathcal{B}$-covering $\mathcal{U}_1$ for the open $U_1$.

        \item In general, at the inning $m$ of the game, Player I chooses the open $V_m\in \mathcal{U}_{m-1}$ that contains the point $(\varepsilon, 0)$. As $V_m=W_m\times Z_m$ has the point $(\varepsilon, 0)$ then $W_m=A_{F_{k_m}}$ for some $F_{k_m}\subset [\omega\setminus \lbrace 0 \rbrace ]^{<\aleph_0}$ such that $F_{s_{m-1}}\subset F_{k_m}$. After that Player I moves with an open $U_n=\Omega (\stackrel{\circ}{\lceil t_{n_{m}}\rceil}, \lfloor t_{n_m}\rfloor)\times A_{F_{s_m}}\subset V_1$ such that $t_{n_{m-1}}< t_{n_m}$ and $F_{k_m}\subset F_{s_m}$. Player II responds with a pairwise disjoint open $\mathcal{B}$-cover $\mathcal{U}_m$ for the open $U_m$.
    \end{itemize}

    At the we have the play $$\bigcap_{m\in\mathbb{N}} U_m= \bigcap_{m\in\mathbb{N}} \Omega (\stackrel{\circ}{\lceil t_{n_{m}}\rceil}, \lfloor t_{n_m}\rfloor)\times \bigcap_{m\in\mathbb{N}} A_{F_{s_m}} = \lbrace \varepsilon \rbrace \times \omega_1 \setminus \bigcup_{m\in\mathbb{N}}F_{s_m}$$
    which contains $\aleph_1$ points but no open subset. It cannot be written as the union of a point and an open, and thus $I\uparrow End_{\mathcal{B}}(X)$. By  Theorem \ref{thm3} $X$ is not an end space.
\end{proof}

In particular, there \textbf{cannot} be a binary $\otimes$ over special trees such that $\Omega(T_1 \otimes T_2) = \Omega(T_1)\times \Omega(T_2)$. Observe the above examples are also edge-end spaces, and every edge end-space is an end-space so the above product also proves the edge end-spaces are not closed under topological product.

\section*{Acknowledgments}
The first named author thanks the support of Fundação de Amparo à Pesquisa do Estado de São Paulo (FAPESP), being sponsored through grant number 2025/12199-3. The second named author acknowledges the support of CAPES through grant number 001. The third named author thanks the Italian National Group for the Algebraic and Geometric Structures and their Applications (GNSAGA-INdAM) for their support. The fourth named author acknowledges
the support of Conselho Nacional de Desenvolvimento Científico e Tecnológico (CNPq) through grant number 165761/2021-0.
\bibliography{bibliography}
\bibliographystyle{plain}    

\Addresses

\end{document}